\documentclass[12pt,oneside,a4paper]{amsart}

\usepackage{TDefn}
\newcommand{\sat}{{\rm sat}}
\newcommand{\cts}{\textnormal{top}}
\newcommand{\hol}{{\mathord{\mathrm{hol}}}}
\newcommand{\antihol}{\overline{\textrm{hol}}}
\newcommand{\alg}{{\mathord{\mathrm{alg}}}}

\title{Homology of spaces of curves on blowups}

\author{Ronno Das}
\address{Department of Mathematics, Stockholm University}
\email{\href{mailto:ronnodas@gmail.com}{ronnodas@gmail.com}}

\author{Philip Tosteson}
\address{Department of Mathematics, University of North Carolina}
\email{\href{mailto:ptoste@unc.edu}{ptoste@unc.edu}}

\begin{document}

\begin{abstract}
We consider the space of holomorphic maps from a compact Riemann surface to a projective space blown up at finitely many points.  We show that the homology of this mapping space equals that of the space of continuous maps that intersect the exceptional divisors positively, once the degree of the maps is sufficiently positive compared to the degree of homology. The proof uses a version of Vassiliev's method of simplicial resolution.  As a consequence, we obtain a homological stability result for rational curves on the degree $5$ del Pezzo surface, which is analogous to a case of the Batyrev--Manin conjectures on rational point counts.  
\end{abstract}

 \maketitle

	\tableofcontents

\section{Introduction}

Let $X$ be a projective algebraic variety and $C$ be a projective algebraic curve over $\bbC$. What is the homology of the space of algebraic (or equivalently holomorphic) maps $\Alg(C,X)$ from $C$ to $X$?  

Segal \cite{segal1979topology} discovered a remarkable phenomenon:  when $X = \bbP^n$,  the homology of the space of degree $d$ algebraic maps approximates the homology of the space of degree $d$ continuous maps as $d \to \infty$.   More precisely,  for all $i \in \bbN$ and $d \gg 0$ the map $$H_i(\Alg_d(C,X)) \to H_i(\Top_d(C,X)),$$ induced by inclusion is an isomorphism.  Here $\Alg_d$ (resp.\ $\Top_d$) denotes the clopen subset consisting of degree $d$ algebraic (resp.\ continuous) maps.   Since Segal's work this result has been extended to rational curves on  toric varieties,  Grassmannians, flag varieties, and more generally varieties with a dense solvable group action \cite{kirwan1985spaces, guest1995topology, boyer1999stability, boyer1994topology, Gravesen}.   Motivated by these results, Cohen--Jones--Segal \cite{cohen-jones-segal} made a conjecture about this phenomenon holding more generally.

In this paper, we consider the case where $X$ is the blowup of a projective space $\bbP(V)$ at a finite number of points $p_1, \dots, p_r$. (Throughout, we take $\dim V \geq 3$ since otherwise the blowup is trivial).   This case is notable for two reasons.   First, when $r$ is large and the points are in general position, $X$ is not a homogeneous variety (its automorphism group is discrete), whereas all previous instances concerned homogeneous targets.  Second, it includes all $2$-dimensional Fano varieties (del Pezzo surfaces).  

Classes in $H_2(X)$ are parameterized by tuples $(d, n_1, \dots, n_r) \in \bbZ^{r+1}$, where $n_i$ records the intersection number with the exceptional divisor $E_i$ and $d$ records the intersection number with the hyperplane class $H$.  We have decompositions $$ \Alg(C,X) =\bigsqcup_{(d,n) \in \bbZ^{r+1}} \Alg_{d,n}(C,X), \qquad \Top(C,X) =\bigsqcup_{(d,n) \in \bbZ^{r+1}} \Top_{d,n}(C,X), $$  where the summand associated to $(d,n)$ consists of maps $f: C \to X$ such that $f_*[C] = (d,n)$.  By composing with the blowdown map, the space $\Alg_{d,n}(C,X)$ of algebraic maps of multi-degree $(d,n)$ can be identified with the space of degree $d$ algebraic maps $C \to \bbP(V)$ that intersect $p_i$ with multiplicity $n_i$. 

Our most general result relates  $\Alg_{d,n}(C,X)$ to a variant of the continuous mapping space consisting of maps with that intersect $E_i$ \textbf{positively}.  Let $\Top^+(C,X) \subseteq \Top(C,X)$ be the subspace of continuous maps $f$ such that  
\begin{itemize}
\item $f\inv(E_i)$ is discrete and
\item $f$ has positive local intersection multiplicity (see \cref{intersection-multiplicity}) with $E_i$ at every point of $f\inv(E_i)$.
\end{itemize}
As above we have a decomposition $\Top^+_{d,n}(C) = \bigsqcup_{d,n} \Top_{d,n}^+(C,X).$ 
We have the following theorem, which we state for both pointed and unpointed mapping spaces.  When $X$ and $C$ have a distinguished base point, we will write $\Top_*(C,X)$, (resp.\ $\Alg_*(C,X)$ etc.), for the space of pointed continuous (resp.\ algebraic etc.) maps.

\begin{thm}\label{thm:blowupThm}
Let $X = \Bl_{p_1, \dots, p_r} \bbP(V)$ for any $r \ge 1$.
Let $\alpha = (d,n_1, \dots, n_r) \in \bbZ_{\ge 0}^{r+1}$ and set $M_{\alpha}:= d-\sum_{i = 1}^r n_i$. If $M_\alpha > 0$, then the maps $$H_i(\Alg_{k\alpha}(C,X)) \to H_i(\Top_{k\alpha}^+(C,X)), \quad H_i(\Alg_{k\alpha,*}(C,X)) \to H_i(\Top_{k\alpha,*}^+(C,X))$$ are isomorphisms for $i < M_\alpha k - 2g - 2$.  
 
If the points $\{p_j ~|~ j \leq \dim V\}$ are in linearly general position 
then the same statement holds, but with $M_{\alpha}$ replaced by the improved constant $\tilde M_{\alpha} := d -\sum_{i =1}^r n_i +\max_{j \leq \dim V} n_j$.  (Here we define $n_{j} = 0$ if $r < j \leq \dim V$.)
\end{thm}

It is natural to ask to whether the above theorem holds when $\Top^+$ is replaced by $\Top$.  In general the answer is no: the homology of the space of positive maps differs from that of all continuous maps.  For instance, when $n = 0$ we have that  $\Top^+_{d,0}(C, X)$ can be identified with the space of continuous maps $C \to \bbP(V) - \{p_1, \dots, p_r\}$ that have degree $d$ as a map to $\bbP(V)$. However, we suspect that for $n \gg 0$ the homology of the space of positive maps approximates the homology of the space of continuous maps, see \cref{posmapsquestion}. 
    
In the pointed case with $C = \bbP^1$ we may pass from $\Top^+$ to $\Top$ provided that each $n_i > 0$.
\begin{cor}\label{non-positivecorollary}
   Let $X = \Bl_{p_1, \dots, p_r} \bbP(V)$ as in \cref{thm:blowupThm}, but with $\{p_j ~|~ j \leq \dim V\}$ in linearly general position.    Let $\alpha = (d,n_1, \dots, n_r)$  with $n_i > 0$ for all $i = 1, \dots r$.  Further suppose that $d> -n_j + \sum_{i = 1}^r n_i$ for all $j \leq \dim V$ (again we set $n_j = 0$ for $r < j \leq \dim V$).  Then the map \[H_i(\Alg_{k\alpha,*}(\bbP^1,X)) \to H_i(\Top_{k\alpha,*}(\bbP^1,X))\] is an isomorphism for $k \gg 0$ (more precisely for $i < \tilde M_{\alpha} k - 2g - 2$, where $\tilde M_\alpha$ is defined as in \cref{thm:blowupThm}).
\end{cor}
Note that $\Top_{k \alpha,*}(\bbP^1, X)$ is independent of $\alpha$ (each component of the double loop space of $X$ is isomorphic), so \cref{non-positivecorollary} implies a homological stability statement for the space of algebraic maps.
 
\Cref{non-positivecorollary} is obtained from \cref{thm:blowupThm} by deforming the points to lie on a line, and applying previous results on mapping spaces for varieties admitting a dense solvable group action.   Surprisingly, despite this proof method, the corollary does not hold for points which are not in linearly general position (unless we strengthen the bound to $d > \sum_{i} n_i$ as in \cref{thm:blowupThm}).  For instance, if $\dim V = 3$ and  $p_1, p_2, p_3$ lie on a line and $(d,n)=(5,2,2,2)$ then by Bezout's theorem any degree $5k$ algebraic map $C \to \bbP(V)$ with $\mult(f\inv(p_i)) = 2k$  must lie on the line.  It follows that $\Alg_{*, k \alpha}(\bbP^1,X) = \emptyset$, while the topological mapping space is nonempty.

\subsection{Del Pezzo surfaces and Batyrev--Manin conjectures}

Work by Ellenberg--Venkatesh--Westerland \cite{venkatesh2010statistics,  ellenberg2016homological} has drawn attention to the relationship between the homological stability phenomena sequences of algebraic varieties, and the asymptotic behavior of associated arithmetic statistics.   In our case, if $C$ and $X$ were defined over a finite field $\bbF_q$ then the set of $\bbF_q$-valued points of $\Alg_{d,n}(C,X)$ would correspond to the set of $K$-valued points of $X$ satisfying certain height conditions,  where $K=\bbF_q(C)$ is the function field of $C$.  The large-height asymptotics of the number of $K$-valued points on a Fano variety $X/K$ (for $K$ a global field) is the subject of conjectures due to Batyrev--Manin \cite{BatyrevManin}.  

   In dimension $2$, these conjectures concern del Pezzo surfaces. Over $\bbC$, a del Pezzo  surface $X$ of degree $\deg(X)$ is isomorphic to the blowup of $\bbP^2$ at $9-\deg(X)$ general points (or $\deg X = 8$ and $X = \bbP^1 \times \bbP^1$). In this case, the Batyrev--Manin conjectures suggest that there is a subset of the ample cone  $U \subseteq A  \subseteq H_2(\bbR)$ that is arithmetically dense in the sense that $$\lim_{R \to \infty} \frac{\#(U \cap H_2(X, \bbZ) \cap B(0, R))}{ \#(A \cap  H_2(X, \bbZ) \cap B(0, R))} = 1$$ and such that for every class $\alpha \in U \cap H_2(X, \bbZ)$ the homology of $\Alg_{k \alpha}(C,X)$ approximates the the homology of $\Top_{k\alpha}(C,X)$ for $k \gg 0$.    
   
   \begin{remark}
        The reader may wonder why the \emph{a priori} transcendental $\Top(C, X)$ should appear in an analog of the arithmetic Batyrev--Manin conjectures. A full justification is outside the scope of this paper, but as motivation we note that (1) $\Top(C,X)$ appears as the limiting object in previous homological stability results for algebraic mapping spaces, and (2) its rational homology admits purely algebraic models \cite{haefliger}.
   \end{remark} 
   
   In this direction, we have the following result for $\deg(X) = 5$:

\begin{thm}\label{thm:delPezzo}
	Let $X$ be the degree $5$ del Pezzo surface, and let $\alpha = (d, n_1, n_2, n_3, n_4)$ be an ample class satisfying the condition that $\min_{i = 1}^4(n_i)$ is distinct from the other values of $n_i$. Then the map $$H_i(\Alg_{k \alpha,*}(\bbP^1, X)) \to H_i(\Top_{k \alpha,*}(\bbP^1, X))$$ is an isomorphism for all  $k \gg 0$  (more precisely for $i < N_\alpha k -2g -2$ with $N_{\alpha} := \max_{\sigma \in S_5} \tilde M_{\sigma \alpha}$, where the symmetric group $S_5$ acts on $\alpha$ via its usual action on $X$ by Cremona transformations and  $\tilde M_{\alpha}$ is as defined in \cref{thm:blowupThm}). 
\end{thm}

We view \cref{thm:delPezzo} as a homological analog of the Batyrev--Manin conjectures for the degree $5$ del Pezzo surface, proved by de La Bret\'eche \cite{de2002nombre} in the case of a split surface over $\bbQ$.  Our methods also work for higher degree del Pezzo surfaces (the toric cases). 

For lower degree del Pezzo surfaces, \cref{non-positivecorollary} still applies to prove that for certain classes $\alpha$ the map $$H_i(\Alg_{k \alpha,*}(\bbP^1, X)) \to H_i(\Top_{k \alpha,*}(\bbP^1, X))$$ is an isomorphism.  However, because the ample cone of a lower degree del Pezzo surface is larger, the classes for which \cref{non-positivecorollary} applies are not dense in the ample cone.

	\subsection{Proof strategy} 
            Our approach to establishing \cref{thm:blowupThm} is through a version of Vassiliev's method for computing the cohomology of discriminant complements, which is a topological version of an inclusion-exclusion argument.  As mentioned above,  $\Alg_{d,n}(C,X)$ is identified with the space of degree $d$ maps to $C \to \bbP(V)$ passing through $p_i$ with multiplicity $n_i$. Thus there is a principle  $\bbC^*$-bundle  $\tilde \Alg_{d,n}(C,X) \to \Alg_{d,n}(C,X)$  that parametrizes the data of a line bundle $L \in \Pic^d(C)$ and a nonvanishing algebraic section $s \in \Gamma(L \otimes V)$ that intersects $L \otimes \ell_i$ with multiplicity exactly $n_i$.  
            
            Let $W_n \subseteq \prod_{i = 1}^{r} \Sym^{n_i} C$ be the space of pairwise disjoint divisors.  Then there is a map $\tilde \Alg_{d,n}(C,X) \to W_n \times \Pic^d(C),$ taking a section $s \in \Gamma(V \otimes L)$ to the pair $(s\inv(\ell_i))_{i = 1}^r$,~$L$.  In fact, $\tilde \Alg_{d,n}(C,X)$ is an open subset of the linear space over $E_{d,n} \to W_n \times \Pic^d(C)$ whose fiber over $(U_i)_{i = 1}^r, L$ consists of sections $s \in \Gamma(L \otimes V)$ such that $s(U_i) \subseteq \ell_i$.  When $d$ is large enough relative to $n$, these fibers have constant rank and $E_{d,n}$ is a vector bundle.  The complement of $\tilde \Alg_{d,n}(C,X)$ is a closed subset of $\Delta_{d,n} \subseteq E_{d,n}$ consisting of sections that vanish or intersect some $\ell_i$ to greater order.

            We stratify $E_{d,n}$, using a poset $\Hilb(C)^{Q_r}$ whose elements are collections of divisors $D_{\ell_1}, \dots, D_{\ell_r}, D_0 \in \Hilb(C)$, with $D_0 \subseteq D_{\ell_i}$ for all $i$.  Given such an element, we can consider the space of sections $s \in \Gamma(L \otimes V)$ satisfying the incidence conditions $s(D_i) \subseteq \ell_i$ and $s(D_0) \subseteq 0$.  As $D_{\ell_i}, D_0$ vary, the subspaces of sections satisfying these incidence conditions form a vector bundle over an appropriate configuration space (at least if $d \gg 0$).

            We use this stratification to produce a bar complex that resolves the ``discriminant locus'' $\Delta_{d,n}$ in the sense that it comes with a map to $\Delta_{d,n}$ that induces isomorphisms on homology in a range of degrees (governed again by how large $d$ is compared to $n$).  This resolution produces a spectral sequence that converges to the compactly supported cohomology of $\Delta_{d,n}$ and whose terms are the compactly supported cohomology of vector bundles over configuration space.

            We then perform an analogous construction for a `semi-topological' version of $\Alg_{d,n}(C,X)$. Specifically, we introduce a moduli space $\cM_{d,n}$, which is intermediate between the space of holomorphic and continuous maps.  The points of $\cM_{d,n}$ parameterize isomorphism classes of the following data: 
            \begin{itemize}
                \item a degree $d$ holomorphic line bundle $L$ on $C$
                \item a collection of disjoint divisors $D_i \in \Sym^{n_i}(C), i = 1, \dots, r$
                \item a \textbf{continuous} section $s \in \Gamma_{\rm top}(C, L \otimes V)$
            \end{itemize}
            such that if $p \in C$ is a multiplicity $k$ point of $D_i$, the section $s$ intersects $\ell_i$ holomorphically to order exactly $k$ in the following sense:
            \begin{itemize}
                \item the section $\overline s \in \Gamma(L \otimes V/\ell_i)$ takes the form $s = az^k + o(|z|^k)$, where $z$ is a holomorphic local coordinate of $L$ at $p$  and $a \in V/\ell_i - 0$.
            \end{itemize}
            Taking a similar approach to the homology of $\cM_{d,n}$, we show that $\Alg_{d,n}(C,X) \to \cM_{d,n}$ induces an isomorphism on homology in a range of degrees (depending on $d,n$).  Then we show that $\cM_{d,n}$ is weakly homotopy equivalent to $\Top^+_{d,n}(C,X)$, in order to establish Theorem \ref{thm:blowupThm}.

            \subsection{Summary of paper}
            In \cref{sec:review} we review Gysin maps, stratifications and bar constructions.  In \cref{sec:Hilb}, we introduce the poset $\Hilb(C)^Q$ of divisors that we use to resolve the complement of $\tilde \Alg_{d,n}(C,X)$.   In \cref{combinatorialtypes}, we introduce a stratification of $\Hilb(C)^Q$.  In \cref{stratificationConvergence}, we use this stratification to compute the homology of bar constructions over $\Hilb(C)^Q$.  We also prove an important criterion for the homology of a simplicial resolution to agree with the homology of the discriminant locus, \cref{BarApproximation}.  In \cref{sec:SemiTop}, we describe how to topologize $\cM_{d,n}$ and similar function spaces.  We also construct finite dimensional approximations to these function spaces, in order to be able to apply compactly supported cohomology.  In \cref{sec:Semitopcompare}, we use these finite dimensional approximations and \cref{BarApproximation} to establish \cref{semitopcompare}: a criterion for the map $\Alg_{d,n}(C,X) \to \cM_{d,n}$ to induce an isomorphism on a range of homology groups.  In \cref{sec:posmapscompare}, we show that semi-topological model is weakly homotopy equivalent to the space of positive maps.   In \cref{sec:unobstructedness}, we show that certain spaces of sections are unobstructed in order to verify the hypotheses of \cref{semitopcompare} and deduce our main results.

	\subsection{Relation to other work}

        All previous results we know, showing that the integral homology of  $\Alg_*(\bbP^1,X)$ stabilizes to  $\Top_*(\bbP^1,X)$ concern varieties $X$ that are \emph{homogeneous} in the sense that they admit an action of a solvable algebraic group $N$ with a dense orbit.  In these cases, there is a description of the algebraic mapping space in terms of certain labelled configuration spaces of zeroes and poles.  In our case, when $r \geq \dim V +2$ there is no analogous description. (Configuration spaces play an important role in our argument too, but they appear in a different way). 
        On the other hand, work of Browning--Sawin \cite{BrowningSawin} on a geometric version of the circle method yields a way to compare the rational homology of a space of pointed algebraic maps to low-degree affine hypersurfaces (which are not homogeneous) to the homology double loop spaces.

        There have been a number of applications of variants of the Vassiliev method to spaces of algebraic/holomorphic maps to homogeneous targets (including cases where the source variety is higher dimensional) by Mostovoy \cite{mostovoy2012truncated}, Koszlowski--Yamaguchi \cite{kozlowskiyamaguchiToric} and Banerjee \cite{banerjee2022moduli}.
        
        Aumonier has used the Vassiliev method to establish 
        a general h-principle that relates holomorphic sections of a vector bundle satisfying incidence conditions to continuous sections of a jet bundle (satisfying the same incidence conditions) \cite{aumonier}. Our results do not fit into this framework.  Although the fiber of $\tilde \Alg_{d,n}(C,X)$ above $L \in \Pic^d(C)$ equals the space of holomorphic sections $L \otimes V$ satisfying incidence conditions, these conditions are not a locally closed subset of the jet bundle $J(L \otimes V)$.  In fact, sections of $J(L \otimes V)$ satisfying these incidence conditions differ from the semi-topological model $\cM_{d,n}$ that we compare $\Alg_{d,n}(C,X)$ with.  (It is for this reason that we introduce $\cM_{d,n}$,  see \cref{remark:Whysemitop?} for further discussion).   Additionally, to obtain a stability range sharp enough to conclude \cref{thm:delPezzo}, we need a weaker condition than jet ampleness.

        Lehmann--Tanimoto have introduced a geometric analog of Batyrev--Manin's conjecture \cite{lehmann2019geometric}, concerning the number of irreducible components of the space of rational curves (related to degree $0$ homology).  In the del Pezzo surface case, this conjecture holds by the work of Testa \cite{Testa}.  From the point of view of the analogy with the arithmetic conjectures, stability for higher order homology groups is related to the constant that appears in Peyre's version of the Manin conjectures \cite{Peyre}.

        \subsection{Further questions}

        \subsubsection{Lower degree del Pezzo surfaces}  As mentioned above, our approach only yields partial information about degree $\leq 4$ del Pezzo surfaces.  For instance, we do not know the answer to the following question.   For a degree $4$ del Pezzo, which is isomorphic to $\bbP^2$ blown up at $5$ general points,  with $\alpha$ the anti-canonical class  $(3,1,1,1,1,1)$ is the map $$H_i(\Alg_{k \alpha,*}(\bbP^1,X)) \to H_i(\Top_{k \alpha,*}(\bbP^1, X))$$ is an isomorphism for $k \gg 0?$ (More generally, we can consider the same question for classes $\alpha$ lying in the cone spanned by an $\epsilon$-neighborhood of the anti-canonical class).  

        \subsubsection{Positive mapping spaces} \label{posmapsquestion} 
        
        Motivated by the appearance of the space of positive maps in \cref{thm:blowupThm},  we  ask the following question.   Let $M$ be a manifold, and let $E \subseteq M$ be a codimension-$2$ submanifold with connected components $E_1, \dots, E_r$, equipped with an orientation of its normal bundle $N_E$.
	
	Let $C$ be an oriented surface.   We consider the space $\Top^+(C, M)\subseteq \Top(C,M)$  of maps $f: C \to M$ with $f\inv(E)$ discrete and all local intersection multiplicities positive.   Given $n \in \bbZ^{r}$ we let $\Top_{n}(C, M)$ denote the space of maps $f: \Sigma \to M$ such that $f_*(c_1(N_{E_{i}})  \cap [C] = n_i$.   (When $C$ is not compact, by convention $[C] = 0$ so this intersection number is always zero.)
		
	\begin{itemize}
			\item Does the inclusion $\Top^+_n(C, M) \to \Top_{n}(C, M)$ induce an isomorphism on $H_i$ for all $n \gg 0$?  
	\end{itemize}
	Here by $n \gg 0$ we mean that $n_i \gg 0$ for every component $n_i$ of $n$.  A positive answer would be interesting, even in the specific case of a blowup of projective space, as it would immediately yield a generalization of \cref{non-positivecorollary} to the unpointed and higher genus case.

    \subsubsection{Higher genus maps to solv-varieties} 
        To obtain \cref{non-positivecorollary}, we apply the results of \cite{boyer1999stability} relating the spaces $\Alg_{*}(\bbP^1, X) \to \Top_*(\bbP^1, X)$, whenever $X$ admits an action by a solvable algebraic variety with dense orbit.  Is it possible to adapt these methods to the case where $\bbP^1$ is replaced by a higher genus curve, or to the unpointed case?  The generality of Segal's results in the case $X = \bbP^n$ suggests a positive answer,  which would yield a generalization of \cref{non-positivecorollary} to the unpointed and higher genus case.

    \subsubsection{Arithmetic and motivic analogs} 
        Our approach to computing the homology of $\Alg_{d,n}(C,X)$ (the Vassiliev method) is a categorification of a strategy for counting the number of points $\Alg_{d,n}(C,X)(\bbF_q)$ (an inclusion-exclusion or sieve argument).   Compared with de La Br\'eteche's argument establishing the Batyrev--Manin conjecture for a split degree 5 del Pezzo over $\bbQ$, it is surprising to us that a direct inclusion-exclusion approach to \cref{thm:delPezzo} is successful.  

        Can the inclusion-exclusion approach be used to establish the Batyrev--Manin conjectures for other degree $5$ del Pezzo (either over function fields or number fields)?  The main difficulty in adapting our approach lies in bounding the error term arising from truncating the inclusion-exclusion term at finite level.  In the setting considered in this paper, it suffices to bound the dimension of the homological contribution of the error term--so we expect our arguments to translate most easily to a ``motivic'' setting where we consider convergence of the classes $[\Alg_{d,n}(C,X)]$ in the Grothendieck ring of varieties with respect to a dimension filtration.

        \subsection{Acknowledgements} The authors would like to thank Alexis Aumonier,  Sylvain Douteau, Benson Farb, Peter Haine, Brian Lehmann, and Will Sawin for discussions related to the topic of this paper. We also thank Sho Tanimoto for helpful and detailed comments on a previous version of the paper.
        
        RD was supported during parts of the project by the European Research Council (ERC) under the European Union’s Horizon 2020 research and innovation programme (grant agreement No.\ 772960) and by the Danish National Research Foundation through the Copenhagen Centre for Geometry and Topology (DNRF151).

        PT was supported during parts of the project by NSF grant DMS-1903040.
			
	\section{Gysin maps and poset topology} \label{sec:review}

     \subsection{Gysin maps}  \label{subsec:Gysin}

For a locally compact topological space $X$,  we use $H^*_c(X)$ to denote the compactly supported cohomology.  We take $H^*_c(X)$ to be defined in terms of sheaf cohomology:  given any injective resolution of sheaves on $X$,   $\bbZ_X \simto \cF$, the compactly supported cohomology of $X$ is computed by  $H^*(\pi_! \cF)$,  where $\pi_!$ denotes the compactly supported global sections.  For a reference on sheaf cohomology and standard facts about operations on derived categories of sheaves see e.g. Iversen \cite{Iversen}.
We will write $C^*_c(X; \cF) = \pi_! \cF$ for the chain complex of compactly supported cochains (when the resolution is unspecified, the statements will be independent of the choice of resolution).

For a finite filtration by closed subspaces $$Z_0 \subseteq Z_1 \subseteq \dots \subseteq Z_n = X,$$  the induced co-filtration $$\cF \to i_{n-1 !} \cF|_{Z_{n-1}}  \to \dots \to i_{0 !} \cF|_{Z_{0}}$$ gives rise to a spectral sequence with $E_1$ page  $$\bigoplus_{i = 0}^N H^*_c(Z_i,Z_{i-1})  = \bigoplus_{i = 0}^N H^*_c(U_i),$$ converging to $H^*_c(X)$.  Here $U_i:= Z_i - Z_{i-1}$, and the identity $H^*_c(U_i)= H^*_c(Z_i,Z_{i-1})$ follows from the exact triangle $j_{U_i!} j_{U_i}^* \bbZ \to \bbZ_{Z_i} \to i_{Z_{i-1} *} i_{Z_{i-1}}^* \bbZ$.

Now suppose that $i:Y \subseteq X$ is a closed embedding,  and $\theta \in H^t(X,X-Y)$  is a class.  Following Fulton--Macpherson \cite{fulton-macpherson},  we think of $\theta$ as a generalized orientation class for $i$, to which we may associate a Gysin map as follows.  We have that $$H^t(X,X-Y) = H^t(R\pi_* Ri^! \bbZ_X) = \Ext^t_{\Sh(X)}(i_! \bbZ_Y, \bbZ_X),$$ where $R\pi_*, Ri^!$ denote functors on derived categories and for the first equality we use the exact triangle $i_! Ri^!\bbZ_{X} \to \bbZ_X \to Rj_{*} j^* \bbZ_Y$ and the adjointness of $i_!, Ri^!$.  Therefore $\theta$ induces a map in the derived category of sheaves on $X$ $$\theta: i_! \bbZ_Y \to  \bbZ_X[t].$$   Choosing an injective resolution $\bbZ_X \simto \cF$,  we may choose a representative  $\tilde \theta: i_!\cF|_Y \to \cF[t]$.  Applying $\pi_!$  we obtain a Gysin map  $\theta_!: H^*_c(Y) \to H^{*+t}_c(X)$, independent of the choice of $\tilde \theta$.
Considering the filtration $Z_i \cap Y$,  the representative $\tilde \theta$ induces a map of co-filtered sheaves,  so we obtain a  Gysin map of spectral sequences:  $$\theta_!:\bigoplus_{i = 0}^N H^*_c(U_i \cap Y) \to \bigoplus_{i = 0}^N H^{*+t}_c(U_i ).$$ 
In particular, if $\theta$ induces an isomorphism  $H^*_c(U_i \cap Y) \to H^*_c(U_i)$  for all $i$,  then we have $\theta_{!}: H^*_c(Y) \iso H^{*}_c(X)[t]$.

More generally, suppose we are given a poset $P$ and collection of closed subsets $Z_{p} \subseteq X$, for $p \in P$ satisfying $Z_p \supseteq Z_q$ for $p \leq q$. 
\begin{defn} We refer to the data of $P$ and $\{Z_{p}\}_{p \in P}$ as a \emph{stratification of $X$}.  The \emph{closed strata} are the sets $Z_p$, and the  \emph{locally closed strata} associated to $p \in P$ are the sets $S_p := Z_p - \bigcup_{q > p} Z_q$. 
\end{defn} 
 Assume that $P$ is finite, and choose an injective homomorphism of posets $f:P \to (\bbN,{\leq})$. From $f$, we obtain an associated filtration $$Z_{0} \supseteq \dots \supseteq Z_N, \qquad Z_i := \underset{p, ~ f(p) \geq i} \bigcup Z_p.$$ Consider the associated compactly supported cohomology spectral sequence.   %

Assume that for every $x \in X$, the set $\{p ~|~ p \in Z_x\}$ has a maximal element (this implies that $S_p \cap S_q = \emptyset$ for $p \neq q$).   Then $Z_i - Z_{i-1} = S_{f \inv(i)}$ and so we obtain a spectral sequence  with $E_1$ page $\bigoplus_{p} H^*_c(S_p) $, converging to $H^*(X)$.   Again, a closed embedding $i:Y \to X$ and orientation class $\theta \in H^*(X,X-Y)$ induces a Gysin map of spectral sequences.  

Without the assumption that $\{p ~|~ p \in Z_x\}$ has a maximal element, $Z_i - Z_{i-1} \subseteq S_{f \inv(i)}$ is an open subset  and the spectral sequence takes the form $\bigoplus_{p} H^*_c(S_p')$ where $S_p' \subseteq S_p$ is an open subset (depending on the choice of ordering $f$).

The following compatibility between Poincar\'e duality and Gysin maps is standard (see e.g.\ \cite[Problem~11-C]{milnor-stasheff}), we state it here for convenience.
\begin{prop}\label{gysinPD}
	Let $E \subseteq F$ be an inclusion of complex vector bundles over a manifold $M$. Write $\dim F$, $\dim E$ for the real dimension of the total spaces of $F,E$ respectively.  Let $\theta \in H^{\dim F - \dim E}(F, F-E)$ be the associated Thom class.  Then for every open subsets $U \subseteq F$, the diagram 
\[
\begin{tikzcd}
{H_i(U \cap E)} \arrow[d] \arrow[r] & { H_i(U)} \arrow[d] \\
{H^{\dim E - i}_c(U \cap E)} \arrow[r, "\theta_!"]      & { H_c^{\dim F - i}(U)  }          
\end{tikzcd}
\]
commutes, where the vertical arrows are the Poincar\'e duality isomorphisms.
\end{prop}

\subsection{Topological posets and bar complexes}
In this section, we recall several facts about simplicial resolutions arising from  topological posets.  

	A \emph{topological poset} is a topological space $P$ and a relation $\cR_P \subseteq P \times P$,  that is  reflexive, antisymmetric, and transitive.  We say that the relation $\cR_P$ is \emph{proper} if both projections $\cR_P \to P$ are proper.

A \emph{$P$-space}, consists of a continuous map $Z \to P$  together with continuous map of spaces over $P$ $$a: \cR_P \times_{\pi_2} Z  \to Z$$ where $\pi_2: \cR_P \to P$ denotes the map $(p \leq q) \mapsto q$  and $ \cR \times_{\pi_2} Z$ is considered as a space over $P$ via  $\pi_1$.   We write $Z_p$ for the fiber of $Z$ over $p \in P$.

Given a $P$-space, we may form the bar construction:

\begin{defn} Let $P$ be a topological poset such that the diagonal $\Delta_P$ is open in $\cR_P$.  Let $Z$ be a $P$-space.    The (semi-simplicial) \emph{bar  construction} is  $$\rB_t(Z,P):= \{p_0 < p_1 < \dotsb <p_t \in P, ~z \in Z_{p_t}\} =  (\cR_P  - \Delta_P)^{\times_P (t+1)} \times_{P} Z. $$
Then $\rB_\bullet(Z, P)$ is a semi-simplicial space with boundary maps given by forgetting elements in the chain and applying the action map $a$  when $p_t$ is forgotten.   We write $|\rB(Z,P)|$  for the geometric realization.
\end{defn}

We use topological posets to define a continuous stratifications of a topological space.  Let $X$ be a  topological space, and $P$ be a topological poset.

\begin{defn} \label{MainVassilievConstruction} 
   A \emph{continuous stratification of $X$ by $P$} is a closed subspace $Z \subseteq P \times X$ satisfying $Z_p \supseteq Z_q$ if $p \leq q$.  
\end{defn}

Let $Z \subseteq P \times X$ be a continuous stratification of $X$, and let $W \subseteq P$.  Since $Z$ is canonically a $P$ space,  we may use the bar construction to resolve the complement of the union of strata  $\bigcup_{w \in W} S_w \subseteq Z \times_P W$, in terms of the following poset associated to $W$ and $P$.   

      \begin{defn}  Let $W < P$  be the following topological poset.  We let $$W < P := \{(w, p) \in W \times P ~|~   w < p\} = W \times_{P} (\cR_P - \Delta_P) ,$$ and define the relation \[\cR_{W < P}:= \{ (w, p_1), (w, p_2) ~|~ w < p_1 \leq p_2\} =  (\Delta_W \times \cR_P) \cap  (W < P)^{\times 2}. \qedhere\]  
\end{defn} There is a a morphism of posets $\pi: (W < P) \to P$, and a  map of topological spaces $(W < P) \to W$. 
Pulling $Z$ back to $W < P$, we obtain a $(W < P)$ space, denoted $\pi^* Z$.   Applying the above bar construction, we obtain a semi-simplicial space  $$\rB(W, P, Z) := \rB(W < P, \pi^*Z),$$  together with  an augmentation map $a:\rB(W,P, Z) \to W\times_P Z $ that takes a chain $w < p_1 < \dots < p_r,  z \in Z_{p_r}$ to the pair $w \in W, z \in Z_w$.

	\section{Posets of divisors and section spaces}
	\label{sec:Hilb}

In this section, we introduce a family of  topological posets and use them to stratify spaces of algebraic sections by incidence conditions.  The level of generality taken here is greater than necessary for our main application, but we have found the general framework helpful, and include it with an eye towards later work.  We will include the case relevant for applications as a running example.

\subsection{Divisors labeled by $Q$.}  Let $C$ be an algebraic curve. We write $\overline\Hilb(C)$ for the poset of all closed subschemes of $C$ (including $C$), and $\Hilb(C)$ for the poset of finite length subschemes.  

\begin{defn}
Let $Q$ be a finite poset with top element $\hat 1$.  We put $\tilde Q := Q - \hat 1$  We define  $\overline \Hilb(C)^{Q}$  to be the poset of homomorphisms  $Q \to \overline \Hilb(C)$ that take $\hat 1$ to $C$.  We let $\Hilb(C)^Q$ denote the subposet consisting of homomorphisms such that the preimage of $C$ is precisely $\hat 1$.
\end{defn} 
In other words, an element $x \in \Hilb(C)^{Q}$ consists of a collection of finite closed subschemes $$\{x_q \subseteq C\}_{q \in \tilde Q},$$  satisfying the property  $x_p \subseteq x_q $ if $p \leq q$.   We have that $$\Hilb(C)^{Q} \subseteq \prod_{q \in \tilde Q} \Hilb(C)$$ is a closed subscheme.

 We may equivalently describe an element of $\Hilb(C)^{Q}$  as a function which assigns to each $c \in C$  a family of closed subschemes $\{x_q\}_{q \in \tilde Q}$, which are supported at $c$ and satisfy  the containment relation $x_p \subseteq x_q$ for all $p \leq q$.  Since finite subschemes supported at $c$ are determined by their multiplicity (i.e. $\Hilb(C)$ can be identified with finitely supported functions $C \to \bbN$), this yields a correspondence between points of $\Hilb(C)^{Q}$  and finitely supported functions  $$C \to  \Hom(\tilde Q, \bbN ).$$    Here $\bbN \cup \infty$ is the poset of non-negative numbers and $\Hom$ denotes the poset of (poset homomorphisms $Q \to \bbN$.     
 The \emph{trivial homomorphism} is the homomorphism  
 $q \mapsto 0$ for all $q \in \tilde Q$,  and the \emph{support} of a function $C \to \Hom(\tilde Q, \bbN)$ is the set of $c \in C$  that map to a nontrivial homomorphism.

\subsection{Stratifying sections} \label{generalstratification}
  
In this subsection, we use $ \Hilb(C)^Q$ to continuously stratify spaces of sections over $C$.

 For the remainder of the paper, we will assume that $Q$ has all meets.    Let $\cX \to C \times \cB$ be a  family of schemes of finite presentation over $C \times \cB$, together collection of closed subschemes  $\cK_q \subseteq \cX$ for all $q \in Q$  satisfying: \begin{itemize}
\item $\cK_{\hat 1} = X$ 
\item $\cK_{p \wedge q} = \cK_p \cap \cK_q$  for all $p,q \in Q.$  
\end{itemize}
There is an algebraic space of sections $\Gamma_{\rm alg}(C, \cX) \to \cB$, representing the functor $$(f:T \to B) \mapsto \Gamma(C \times T, f^* \cX).$$  (c.f.\ \stackscite{0CXK}).  
 (In the cases of interest $\Gamma_{\rm alg}(C, \cX)$ will be representable by a scheme).

There are universal algebraic subspaces  $$\cD_q  \subseteq \Hilb(C)^{Q} \times C$$ for $q \in Q$,  which satisfy   $\cD_{q} \subseteq \cD_{q'}$  if $q \leq q'$. (The fiber of $\cD_q$ over $D\in \Hilb(C)^Q$ is the divisor $D_q$).  By restricting sections of $C$ to sections of $D_q$,  we may form the scheme  $\Gamma_Q(C, \cX)$  
parameterizing sections $s \in \Gamma(C, \cX_b)$ such that $s(D_q) \subseteq (\cK_q)_b$  as a fiber product:

$$\begin{tikzcd}
{\Gamma_Q(C, \cX) } \arrow[d] \arrow[r]         & \ \underset{q \in \tilde Q} \prod \Gamma(\cD_q, \cK_q) \arrow[d] \\
{ \Gamma(C, \cX) \times  \Hilb(C)^{Q} } \arrow[r] & {\underset{q \in \tilde Q}\prod\Gamma(\cD_q, \cX)} .           
\end{tikzcd}$$

Note that the underlying topological space of $\Gamma_Q(C, \cX)$ is an $\ \Hilb(C)^{Q} \times B$-space:  if $D_{q} \subseteq D_q'$ for all $q \in Q$ then the set of sections mapping $Z_{q}'$ to $\cK_{q}$ is contained contained in the set of sections mapping $D_q$ to $\cK_q$.  Thus $\Gamma_Q(C, \cX)$ defines a continuous stratification of  $\Gamma(C, \cX)$ by $ \Hilb(C)^{\rQ} \times \cB$, according the preimages of $\cK_q$.

\begin{ex}\label{mainex1}
	 The following is our main example.  Let $V$ be a vector space, and $\ell_1, \dots, \ell_r \subseteq V$ be lines.    Let $\cB = \Pic^d(C)$ and $\cX = \cL_d \otimes V$.   Then $\cX$ is stratified by the poset $Q_r:= \{V,\ell_1, \dots, \ell_r, 0\}$ consisting of the subspaces $\ell_i$, $V$ and $0$ ordered by containment: the stratum corresponding to a subspace $S\subseteq V$ is $S \otimes \cL_d$.
\end{ex}

\subsection{Saturated elements}    Many distinct elements of $\Hilb(C)^Q$ impose the same incidence conditions on a space of sections, and are in a sense ``redundant.''  Using the lattice structure of $Q$, we may construct a smaller poset that removes these redundancies and is combinatorially simpler, keep in mind that $Q$ is finite and assumed to have meets.

\begin{defn} 
We say that an element $(x_q)_{q \in Q} \in \Hilb(C)^{Q}$  is \emph{saturated}  if for every $S \subseteq Q$  the natural containment $$x_{\underset{s \in S}\wedge s}  \subseteq \bigcap_{s \in S}  x_s $$ is an equality.   We write $Q^{\rJ C}$  for the  \emph{subposet of saturated elements}.   Note that $Q^{\rJ C}$ is a Zariski open subset of $\Hilb(C)^Q$.
\end{defn}

We defined an element of $\Hilb(C)^{Q}$ to be saturated if and only if the associated homomorphism $Q \to \Hilb(C)$  preserves meets.    In terms of finitely supported functions $C \to \Hom(\tilde Q, \bbN )$, saturated elements are functions which assign to every $c \in C$  a meet-preserving homomorphism $\tilde g_c:  \tilde Q \to \bbN$.    

We write $g_c: Q \to \bbN \cup \infty$ for the extension of $\tilde g_c$ defined by $g_c(\hat 1) = \infty$.  Because $g_c$ preserves meets, it admits a left adjoint $f_c: \bbN \cup \infty \to Q$  given by $$ f_c(n) := \inf \{ q \in Q ~|~ g_c(q) \geq n \}.$$   Conversely the adjoint $f_c$, which must be join preserving, uniquely determines $g_c$. More precisely we obtain a bijection:
\begin{equation}\label{chain} \begin{multlined}\{  f_c: \bbN \cup \infty \to Q ~| ~ f_c \text{ join preserving and } f_c(\infty) = \hat 1\}\\
\longleftrightarrow \{\tilde g_c: \tilde Q \to \bbN \text{ meet preserving} \} \end{multlined}
\end{equation}
Here we have used that for an adjoint pair $(f_c, g_c)$ the condition that $f_c(\infty) = \hat 1$ is equivalent to $g_c\inv (\infty) = 1$. 

The correspondence \eqref{chain} is inequality reversing:  we have that $f_c \leq f_c'$ if and only if $g_c \geq g_c'$.   It motivates the following definitions.

\begin{defn}
We call a function  $f: \bbN \cup \infty \to Q $ a \emph{chain of $Q$} if $f(\infty) = \hat{1}$ and it preserves joins, i.e. satisfies $f(0) = 0$ and $f(\infty) = f(n) = \hat 1$ for all $n \gg 0$. We write $\Ch(Q)$ for the set of all chains of $Q$.  The \emph{trivial chain} is the largest element of $\Ch(Q)$:  the function $0 \mapsto 0$ and $i \mapsto \hat 1$ for  $i > 0$.  
\end{defn}

We have the following immediate consequence of \eqref{chain}.

\begin{prop}
	There is a canonical bijection  between $Q^{\rJ C}$  and the set of finitely supported functions $C \to \Ch(Q)$.  (Here the support of $f$ is the set of $c$ for which the chain $f(c)$ is nontrivial.)  
\end{prop}

\begin{remark}
	We use the notation $Q^{\rJ C}$  because it is possible to think of the chain $f_c$ as describing the incidence relations of a map from the jet at $c$  to $Q$:  the $i$th element of the chain represents the smallest stratum that $i$th order jet is contained in. 
\end{remark}

\begin{defn} We define the \emph{saturation function} $${\rm sat}: \Hilb(C)^{Q} \to Q^{\rJ C}$$  to be the left adjoint to the inclusion of the subposet of saturated elements. In other words, $\rm sat$ takes  an element of $x \in\Hilb(C)^{Q}$ to the smallest saturated element that is greater than or equal to $x$.  The function $\rm sat$ exists because $Q^{\rJ C}$ is closed under meets.  
\end{defn}

\begin{remark}
	Note that $\sat$ is \emph{not} continuous if $Q^{\rJ C}$ is given the subspace topology as a Zariski open subset of $\Hilb(C)^Q$.  However (for our purposes) it  would be more natural to think of $Q^{\rJ C}$ as a quotient poset of $\Hilb(C)^Q$ (making $\mathrm{sat}$ a quotient map), from which the stratification of the space of sections is pulled back. Then bar constructions over $\Hilb(C)^Q$ would retract onto bar constructions over $Q^{\rJ C}$, simplifying them considerably.  
	
	 Unfortunately, when one attempts to topologize $Q^{\rJ C}$ as a quotient, one obtains a non-Hausdorff space which is difficult to understand.  For this reason, we will work directly with $\Hilb(C)^Q$ but show that (after stratification) bar constructions over $\Hilb(C)^Q$ are closely related to $Q^{\rJ C}$.  This indirect relationship between $Q^{\rJ C}$ and $\Hilb(C)^Q$ is a key technical point for us.
  \end{remark}
 
 \begin{ex}\label{mainex2}
 We continue our main example from \cref{mainex1}, with the poset $Q_r = \{V, \ell_1, \dots, \ell_r, 0\}$.  In this case, an element of $\Hilb(C)^{Q_r}$ consists of a collection of divisors $D_{\ell_1}, \dots, D_{\ell_r}, D_0 \in \Hilb(C)$, such that $D_0 \subseteq D_{\ell_i}$ for all $i$.    An element is saturated if and only if $D_{0} = D_{\ell_i} \cap D_{\ell_j}$ for all $i \neq j$.  The saturation operation replaces $D_0$ by $D_0' = \bigcup_{i \neq j} D_{\ell_i} \cap D_{\ell_j}$ and every other $D_{\ell_i}$ by $D_{\ell_i} \cup D_0'$.  
 
For $c \in C$, the associated function $\tilde g_c: \tilde Q_r \to \bbN$ consists of values $\tilde g_c(\ell_i), \tilde g_c( 0)$ recording the multiplicity of $D_{\ell_i}$ and $D_0$ at $c$.  This function is saturated if and only if $\tilde g_c(0)$ equals all but the maximum of $\{\tilde g_c(\ell_i)\}_{i = 1}^r$.        In this case, suppose that the maximum is achieved by $\ell_j$, and let $m_0 = \tilde g_c(0)$ and $m_{\ell_j} = \tilde g_c(\ell_j) - \tilde g_c(0)$.   The associated chain $f_c : \bbN \cup \infty \to Q_r$ is given by $$f_c(n) = \begin{cases} 0 & \text{ if $n \le m_0$} \\ \ell_j &\text{ if $m_0 < n \le m_{\ell_j} + m_0$} \\ \infty &\text{ if $ m_{\ell_j} + m_0 < n$}.  \end{cases}$$ We denote this chain by $m_0 0 +  m_{\ell_j} \ell_j$ (because on $\bbN_{>0}$, it takes the value $q$, $m_q$ times).   \end{ex}

\subsection{Essential saturated elements}

Because $Q$ has meets, the poset of chains $\Ch(Q)$ has joins:  the join of the chains $q_0 \leq q_1 \dots$ and $p_0 \leq p_1 \dots$ is $p_0 \wedge q_0 \leq \dots$.  (Meets and joins are exchanged because our convention is that the ordering $\Ch(Q)$ is opposite to $Q$).   There is a canonical injective homomorphism $Q\op \to \Ch(Q)$ given by taking $q$ to the chain $q \leq \hat 1 \leq \hat 1 \dotsb$.

\begin{defn} 
Let $f_0 \in \Ch(Q)$ be a chain, and let $S$ be the collection of chains $f_s$ such that $f_0 \prec f_s$, ie such that $f_0 < f_s$ and there is no $f'$ with $f_0 < f' < f_s$.
 We say that a pair of elements $f_0 \leq f \in \Ch(Q)/\Ch(Q)$ is \emph{essential} if $f$ can be obtained from $f_0$ by taking the join of some subset of $S$.  We say that $(w \leq x) \in (W \le Q^{\rJ C})$ is \emph{essential} if it is pointwise essential.  (In other words if for every $c \in {\rm supp}(w \leq x)$ we have that $f_c^w \leq f_c^x$ is essential).  \end{defn}

\begin{ex}\label{essential_example}
	In our main example, we take $Q = Q_r = \{0, \ell_1, \dots, \ell_r, V\}$ and let $W \subseteq \prod_{i = 1}^k \Sym^{k_i} C$ be the set of elements of the product with disjoint support.   We use $\ell_1, \dots, \ell_r$ as our distinguished generators.   Then  a saturated element $(w \leq x) \in (W \le Q^{\rJ C})$ is \emph{essential} if for every $c \in {\rm supp}(w \leq x)$ we have that $f_c^w \leq f_c^x $ is one of the following pairs of chains:
	\begin{itemize}
		\item $d \ell_i \leq (d+1) \ell_i$
		\item $d \ell_i \leq (d-1) \ell_i + 0$
		\item $d \ell_i \leq d \ell_i + 0$
		\item $V \leq \ell_i$ ~(Note that here $V$ denotes the chain $V \leq V \leq \dots$  and $\ell_i$ denotes the chain $\ell_i \leq V \leq \dots$).		\item $V \leq 0$ (if $r \geq 2$) \qedhere
	\end{itemize}
\end{ex}

Our main motivation for introducing essential pairs is the following version of the crosscut theorem.
\begin{prop}\label{essentialinitial}
	For any $f_0 \in \Ch(Q)$, the essential pairs in $f_0 < \Ch(Q)$ form an initial subposet of $f_0 < \Ch(Q)$. In particular $N( f_0 < \Ch(Q))$ deformation retracts onto the nerve of the subposet of essential pairs.
\end{prop}
\begin{proof}
	By definition, we need to show that for any $(f_0, f)$  there is a maximal essential pair $(f_0, \tilde f) \leq (f_0, f)$.  We take $\tilde f$ to be the join of every $f_s$ such that $f \prec f_s \leq f$.  Then the inclusion of the subposet of essential pairs admits a right adjoint, and so the statement follows.  \end{proof}

\subsection{Combinatorial functions on $Q^{\rJ C}$}\label{combinatorialfunctions}
 	 Let $h$ be a function $h: Q \to \bbN$,  satisfying $h(\hat 1) = 0$.    

We may extend $h$ to a function $Q^{\rJ C} \to \bbN$, as follows.  First we may extend $h$ to a function $h: \Ch(Q) \to \bbN$ by defining $h(f) = \sum_{n \ge 1}f(n)$, which is a finite sum since $f(n) = \hat{1}$ for $n \gg 0$.
Next recall that an element $x \in Q^{\rJ C}$ corresponds to a finitely supported function $$C \to \Ch(Q), \qquad c \mapsto f_c, $$       
so we define $h(x) :=  \sum_{ c \in C} h(f_c)$.

We may further extend $h$ to a function on $\Hilb(C)^{Q}$ by precomposing with $\sat$, and extend to pairs of elements by declaring $h(w \leq x) = h(w) - h(x)$.

\begin{ex}\label{multex}
	Let $q_0 \in \tilde Q$, and let $m_{q_0}$ be the indicator function defined by $m_{q_0}(q) =1 $, if $q = q_0$ and $m_{q_0}(q) = 0$ otherwise.   We refer to $m_{q_0}: \Hilb(C)^Q \to \bbN$ as the \emph{multiplicity of $q$}.
\end{ex}

\begin{ex}\label{rankex}
    Let $\rank: \Ch(Q) \to \bbN$ be the rank function, where $\rank(x)$ is the length of the longest chain from $\hat 1$ to $x$.    We extend $\rank$ to a function on $Q^{\rJ C}$ as above.   The value $\rank(x)$ is the length of the maximal chain starting at $\hat 1$ and ending at $x$,  or equivalently the dimension of the simplicial complex $\rN([\hat 1,\sat(x)] \cap Q^{\rJ C})$.    We note that when $Q$ is a \emph{graded poset} in the sense that every maximal chain has the same length,  then the rank function on $\Ch(Q)$ can be obtained from the rank function on $Q$ by the extension procedure above.
\end{ex}

\begin{ex}\label{suppex}
    Given $x \in \Hilb(C)^Q$ we define $\supp(x)$  to be the cardinality of the support of $x$.  This function can be obtained from the above extension procedure from the element of $\Ch(Q)$ that is $0$ on the trivial chain and $1$  otherwise.   The value $|\supp(x)|$ is the complex dimension of the subspace of elements $x' \in Q^{\rJ C}$ with the same \emph{combinatorial type} as $x$ (to be defined in the next section).
\end{ex}

\begin{ex}. Let $Q_r = \{0, \ell_1, \dots, \ell_r, V\}$ be the poset in our main example.  We have two important examples of functions. 

\begin{enumerate}
 	\item We have that $\rank(\ell_i) = 1$ and $\rank( 0) = 2$.  This poset is graded, so these values determine the others.
	\item   Let $\gamma: Q_r \to \bbN$ be the function defined by $\gamma(\ell_i) = 2(\dim V - 1)$ and $\gamma(0) = 2\dim V$.  We extend $\gamma$ to a function on $Q_r^{\rJ C}$.  Heuristically  $\gamma(x)$, is the expected real codimension of the incidence conditions imposed by $x$.  \qedhere
\end{enumerate}

\end{ex} 

\section{Combinatorial stratifications of \texorpdfstring{$\Hilb(C)^Q$}{Hilb(C)\^Q} and \texorpdfstring{$Q^{\rJ C}$}{Q\^{JC}}} \label{combinatorialtypes}

Set theoretically, both  $Q^{\rJ C}$ and $\Hilb(C)^Q$ can be decomposed into a disjoint union of labeled configuration spaces.  To describe these decompositions, we introduce some terminology.  We will use these decompositions to stratify bar constructions in the next section.

\subsection{Combinatorial types} Recall that an element $x \in \Hilb(C)^Q$  corresponds to a finitely supported function $$C \to \Hom(\tilde Q, \bbN) \qquad c \mapsto g_c.$$  The \emph{combinatorial type}  of such a element is the \textbf{multiset} of  functions  $$\type(x):=\{ g_{c_1}, \dots,  g_{c_r}: \tilde Q\to \bbN\},$$ where $c_i \in C$ are the elements such that $ g_{c_i}$ is nontrivial  (i.e.\ not the constant function $0$ on $\tilde Q$).    The \emph{set of combinatorial types} is the set of all multisubsets of functions $g: \tilde Q \to \bbN$.

Note that $x$ is an element of $Q^{\rJ C}$ if and only if each $\tilde g \in \type(x)$ is meet preserving, in which case each $g$ corresponds to a chain $f \in \Ch(Q)$.    In this case, we say the multi-subset, $\type(x)$, is \emph{saturated}.    There is a canonical bijection between saturated types and multisubsets of $\Ch(Q)$; we will use this bijection to specify saturated types.  Furthermore we obtain a saturation function from combinatorial types to saturated types, by  applying $\sat:\Hilb(C)^Q \to Q^{\rJ C}$ pointwise.

\begin{defn} Given a multiset  $T$ of functions $Q \to \bbN$,  we let $\cN_T$ denote the subspace of all $x \in \Hilb(C)^Q$ such that $\type(x) = T$.    \end{defn}

We have that $\cN_T$ is homeomorphic to $\Conf(C; T)$,  the space of configurations of $C$ with labels in the multiset $T$.  And $\cN_T \subseteq Q^{\rJ C}$ if and only if $T$ is a saturated type.

\subsection{Stratifications}
We now define a partial order on the set of combinatorial types, taking into account both the closure relations of $\cN_T$ and the poset structure of $\Hilb(C)^Q$.  Using this order we will stratify $\Hilb(C)^Q$.

\begin{defn} \label{partialorder} 

We write $\{h_1, \dots, h_s\} \leq \{g_1, \dots, g_l\}  $ if there is an injection $G: [s] \to [l]$ such that $g_{G(j)} \geq h_j$ for all $j$. 

We say that   $\{g_1, \dots, g_l\} \geq_{+}  \{h_1, \dots, h_s\}$  if there is a function $F:[l] \to [s]$ such that $  \sum_{i \in F^{-1}(j)} g_i \geq h_j$ for all $j \in [s]$.  ($F$ is necessarily surjective).

We write $\{g_1, \dots, g_l\} \lepe \{h_1, \dots, h_s\}$ if there is a function $F:[l] \to [s]$ such that $  \sum_{i \in F^{-1}(j)} g_i = h_j$ for all $j \in [s]$. 
\end{defn}

The set  $\displaystyle \bigcup_{S \leq_+  T} \cN_S$  is both closed and downward-closed in the poset structure on $\Hilb(C)^Q$.
More precisely, the closure of \(\cN_S\) is directly related to the order \(\lepe\):
\begin{prop} \label{sym-closure-stratification}
If $T \lepe S$ then \(\cN_T \subseteq \overline{\cN_S}\), otherwise \(\cN_T\) is disjoint from \(\overline{\cN_S}\).
\end{prop}
\begin{proof}
It suffices to check this for each coordinate in \(\Hilb^n(C) = \Sym^n(C)\), where it follows from the density of $\Conf^k(C)$ in $\Sym^k(C)$, and the latter being compact.
\end{proof}

There is a further variant of $\leq_+$, which will only be used in \S \ref{subsec:approximationcriteria}.
\begin{defn} \label{closure-downward-closure-order}
We write $\{g_1, \dots, g_l\} \lepd \{h_1, \dots, h_s\}$  if there is a function $F:[l] \to [s]$ such that $  \sum_{i \in F^{-1}(j)} g_i \leq h_j$ for all $j \in [s]$.
\end{defn}

 \begin{defn}
 Given a saturated combinatorial type $S$, we define $$\cS_S:= \displaystyle  \bigcup_{T,~ \sat(T) = S}  \cN_T,$$ where the union is over combinatorial types whose saturation is $S$.   
 \end{defn} 

We note that the subspace topology on $\cS_S$ agrees with the disjoint union topology, $\displaystyle \bigsqcup_{T,~ \sat(T) = S} \cN_T$.   (Because $|\sat(T)| = |T|$, if both $T_1, T_2$ saturate to $S$, no element of $\cN_{T_1}$ is in lies in the closure of $\cN_{T_2}$).

 We define  a new stratification of $\Hilb(C)^Q$  whose locally closed strata are given by $\cS_S$, where $S$ ranges over saturated types.  This stratification is a coarsening of the previous one, obtained by constructing a poset structure on the set of saturated types as follows. 
\begin{prop}\label{posetcollapse}
	There exists a coarsest poset structure 
	 $\leq_{+,\sat}$ on the set of saturated types such that the saturation function from the set of types to the set of saturated types is a morphism of posets  with respect to  
	 $\leq_+$ and $\leq_{+,\sat}$ respectively.
\end{prop}

Motivated by \cref{posetcollapse}, for any saturated type $S$ we define $$\cZ_S := \bigcup_{T \text{ saturated},~ T \leq_{+,\sat}  S} \cS_T \ =  \bigcup_{\tilde T, ~ \sat(\tilde T) \leq_{+,\sat} S} \cN_{\tilde T} .$$  It follows that $\cZ_S$ is both closed and downward closed in $\Hilb(C)^Q$,  because $$\{\tilde T ~|~ \sat(\tilde T) \leq_{+, \sat} S\}$$ is downward closed with respect to $\leq_+$.
 
In addition, we have that $\cZ_{S_1} \subseteq \cZ_{S_2}$ if and only if $S_1 \leq_{+, \sat} S_2$ and $$\cZ_S - \bigcup_{T <_{+,\sat} S} \cZ_T = \cS_S.$$

\begin{proof}[Proof of \cref{posetcollapse}]

We define $\leq_{+, \sat}$ 
 to be the transitive closure of the relation $$S_1 \preceq_{+} S_2 \text{ if there exist types $T_1, T_2$ with $T_1 \leq_+ T_2$ and $\sat(T_i) = S_i$}.$$  To prove that $\leq_{+, \sat}$ is a poset relation, it suffices to show antisymmetry.

Consider the function $|-|$ assigning to a type $T$ the number of elements of $T$.   This function is decreasing and we have that $|\sat(T)| = |T|$.   
  
   Thus given chains of elements $S \preceq_+ \dots \preceq_+ S'$  and $S' \preceq_+ \dots \preceq_+ S$ it follows that $|S| = |S'|$.  So for each $S_1 \preceq_{+} S_2$ in the chain, there exist types $T_{1}, T_{2}$ with $\sat(T_i) = S_i$ and $T_1 \leq_+ T_2$.  We have $|T_1| = |T_2|$  so we have $T_1 = \{e_1, \dots, e_l\} $ and $T_2 = \{h_1, \dots, h_l\}$  and there is a permutation $\sigma \in S_l$ such that $e_{\sigma i} \leq h_i$ for all $i = 1, \dots l$.  Then $\sat(e_{\sigma i}) \leq \sat(h_i)$ for all $i = 1, \dots, l$.  Thus writing $S = \{g_1, \dots, g_l\}$  and $S' = \{g_1', \dots, g_l'\}$, we obtain  permutations such that $\tau$ and $\tau'$  such that $g_{\tau i} \leq g_i'$ and $g_{\tau 'i }' \leq g_{ i}'$ for all $i$, as elements of $\Hom(Q,\bbN)$.  Iterating,  we obtain that  $g_i' \geq g_{\tau i}\geq g_{(\tau'\tau)^n  i}$ for all $n \in \bbN$.  Since $\tau' \tau$ has finite order,  we obtain that $g_i' = g_{\tau i}$  for all $i = 1, \dots, l$.  So $S_1 = S_2$  as desired.  
\end{proof}

\subsection{Relative variant}  We may extend all of the definitions in this section to pairs of elements $$(w \leq x) \in (\Hilb(C)^Q \le \Hilb(C)^Q)$$ by observing that $(\Hilb(C)^Q \le \Hilb(C)^Q) = \Hilb(C)^{Q  \times I}$  where  $I$ is the poset consisting of two elements $\hat 0 \leq \hat 1$ and $Q \times I$ is the product poset.  
For convenience, we expand out these definitions.    The \emph{combinatorial type} of an element $w \leq x$ is  the multi-set  $$\type( w \leq x )=  \{ g^w_{c_1} \leq g^x_{c_1} , \dots,g^w_{c_k} \leq g^x_{c_k} \},$$ of pairs of elements of $\bbN^Q$, where $c_i$ ranges over the elements of $C$ for which $g^x_c$ or $g^w_c$ is nontrivial.   
We obtain a partial order  $\leq_+$  on the set of types,  and a stratification of $(\Hilb(C)^Q \le \Hilb(C)^Q)$ by subvarieties, with open strata $\cN_T$  the set of pairs $w \leq x$  with $\type(w \leq x)= T$.
And we obtain a notion saturated type, and poset structures $\leq_{+, \sat}$  on saturated types  such that $\cZ_S = \displaystyle \bigcup_{\sat(T) \leq_{+, \sat} S} \cN_T$  is both closed and downward closed, and satisfies $$\cS_S =  \cZ_S - \bigcup_{T <_{+,\sat} S} \cZ_{T}.$$

\subsection{Pointed variant}  In dealing with pointed maps from a pointed curve $C$, it is useful to work with a pointed version of combinatorial types where the label of the basepoint is specified.  To that end, we define \emph{pointed combinatorial type} $T$ to be  a pair consisting of $g_0 \in \bbN^{Q}$ and $T'$ an ordinary combinatorial type.   

When $C$ has a distinguished base point $*$ we may associate a pointed combinatorial type to an element $x \in \Hilb(C)^Q$ by taking $g_0 = x_*$ and taking $T$ to be the multiset of labels of points $C-*$.

In this setting, we define the stratum $\cN_T$ associated to to $T$ to be the set of all $x$ such that $\type(x) = T$.   In this case $\cN_T$ is homeomorphic to $\Conf(C-*, T')$.  There is a partial order $\leq_{+}$ on pointed combinatorial types, defined analogously to the partial order on combinatorial types.  We say that $(x_0,\{g_1, \dots, g_l\}) \geq_+ (h_0, h_1, \dots, h_s)$ if there is a map of pointed sets $F: \{0\} \cup [l] \to  \{0\} \cup [s]$ such that $h_i \leq \sum_{i \in F\inv(j)} g_i$.  There an associated poset relation $\leq_{+, \sat}$ and stratification with strata $\cS_T$.

\section{Stratifying bar constructions} 
\label{stratificationConvergence}
 In this section, we will use the stratification introduced in \cref{combinatorialtypes} to prove two important theorems about bar constructions associated to subposets of $\Hilb(C)^Q$.  Throughout, we let $C$ be an algebraic curve, possibly with a chosen basepoint.  We will treat the pointed and unpointed case simultaneously.  In the pointed case, all combinatorial types are assumed to be pointed. 

Let $W \subseteq \Hilb(C)^Q$ be a locally closed union of subsets $\cN_T$ where $T$ is a saturated combinatorial type of $\Hilb(C)^Q$.   Let $P \subseteq (W < \Hilb(C)^Q)$ be a closed and downward closed union of finitely many $\cS_T$ for $T$ a saturated combinatorial type of $W < \Hilb(C)^Q$, such that $P \to W$ is proper.   We say that $T$ is a type of $P$ if $\cS_T \subseteq P$. Let $\cB$ be a locally compact topological space and let $U \subseteq \cB \times W$ be an open subset.  We write $U \le \Hilb(C)^Q$ and $P_U \to U$ for the poset obtained by pulling back $P$ to $U$.  

Suppose we are given a space $E \to U$ with a stratification by  $U\le \Hilb(C)^Q$,  denoted by $Z \subseteq  E \times_U  (U \le \Hilb(C)^Q)$.  We assume this stratification  satisfies the additional condition that the inclusion $Z_{b,w < x} \supseteq Z_{b,w < \sat(x)}$ is an equality for all $b,w,x$.  We write $\rB(U, P, Z)$ for the bar construction with $r$-simplices $$\{(b,w) \in U, w < x_0 < \dots < x_r \in \Hilb(C)^Q ,  s \in Z_{(b,w<x_r)} ~|~ w < x_i \in P \}.$$

\subsection{Spectral sequence}  \label{stratifyingbarconstructions}  
First we will describe a spectral sequence associated to a stratification of the bar construction.  
 
If $T$ is a combinatorial type of $P$ so that $\cZ_T \subseteq P$, we can apply the bar construction to $\cZ_T$ to obtain a closed semi-simplicial subspace $\rB(U, \cZ_T, Z) \subseteq \rB(U,P,Z)$.  Thus we obtain a stratification of the geometric realization $|\rB(U,P,Z)|$ by closed subspaces. Note that an element of $|B(U,P,Z)|$ has a maximal closed stratum that it is contained in,  determined by $\type(x_r)$, where $x_0 < \dots < x_r$ is the chain associated to the element.   Hence this stratification satisfies the assumption of \ref{subsec:Gysin} used to describe the compactly supported cohomology spectral sequence.

Let $\cY_T$ be the locally closed stratum associated to $T$.   A point of $\cY_{T}$ is specified by the following data
\begin{itemize}
    \item an element $(b,w) \in U$,
    \item a point $t$ in the interior of $\Delta^r$
    \item a chain $w < x_0 <  \dots < x_r$ with  $w < x_i \in P$ for all $i$ and $w < x_r \in \cS_T$
    \item an element $s \in Z_{b,w<x_r}$
\end{itemize} 
The space $Z|_{U \times_W \cN_T}$ parameterizes the data of 
\begin{itemize}
    \item $(b,w) \in U$
    \item $(w< x) \in \cN_T$
    \item $s \in Z_{b,w < x}$.
\end{itemize}  There is a map $\cY_T \to Z|_{U \times_W \cN_T}$ taking $(b,w,t, x_\bdot,s)$ to $(b,w,\sat(x_r), s)$.  (The map is well defined because we have assumed that $Z_{b,w<\sat(x_r)} = Z_{b, w< \sat(x)}$.  It is continuous because if $w < x_r \in \cS_T$, then each $\sat(x_i) = \sat(x_r) \in \cN_T$.)

We have that
$\cY_T \to Z|_{U \times_W \cN_T}$ is a fiber bundle and is pulled back from $\cN_T$.  The fiber over $(b,w<x, s)$  only depends on $w<x \in \cN_T$ and is homeomorphic to $$\rN (w, \sat\inv(x)] - \rN (w, \sat\inv(x)).$$ Here $(w,\sat\inv(x)]$ denotes the subposet of $\Hilb(C)^Q$  consisting of elements $x'$ such that $w < x' \leq x$ and $\sat(x') \leq x$.   And $(w, \sat\inv(x)) \subseteq (w,\sat\inv(x)]$ denotes the subposet of $x'$ satisfying the additional condition that $\sat(x') < x$.

\begin{prop}
    The saturation operation $\sat: \Hilb(C)^Q \to \rQ^{\rJ C}$ induces an homotopy equivalence of pairs  $$(\rN(w, \sat\inv(x)]~,{\rN(w, \sat\inv(x))}) \simeq  (\rN((w, x] \cap \rQ^{\rJ C}), ~ \rN((w, x)\cap \rQ^{\rJ C})). $$
    Furthermore, the cohomology of this pair vanishes unless $w< x$ is essential.
\end{prop}
\begin{proof}
    This follows because the saturation operation, being a right adjoint, provides a retract of the first pair onto the second.  
\end{proof}

The bundle of pairs over $\cN_T$, whose fiber over  $(w<x)$ is 
\[(\rN(w, \sat\inv(x)]~,{\rN(w, \sat\inv(x))}),\]
gives rise to a pushforward sheaf of chain complexes on $\cN_T$ whose stalks compute the cohomology of the pair.  We write $\mu(T)[1]$ for this complex of sheaves. 

\begin{remark} We choose this notation because the complex of sheaves $\mu(T)$ is a categorification of the Möbius function of the poset $P$.  More precisely, the Euler characteristic of the stalk of $\mu(T)$ at $w < x$ equals the Möbius function of the interval between $w$ and $x$: we have $\chi(\mu(T)_{w < x}) = \mu(w,x)$, see e.g.\ \cite[Proposition~1.2.6]{wachs}.
\end{remark}

By the discussion above, we have that $H^*_c(\cY_T) = H^*_c(Z|_{U \times_W \cN_T}~,~ \mu(T)[1])$, which vanishes unless $T$ is essential.   Consequently, we obtain the following description of the spectral sequence associated to the stratification of $\rB(U, P, Z)$.  

\begin{thm}\label{VassilievSSDescription}
    There is a spectral sequence converging to  $H^*_c(|\rB(U, P, Z)|)$ with total $E_1$ page $$E_{1}^{*,*} = \bigoplus_{T \text{ essential, saturated}} H^*_c(Z|_{U \times_W \cN_T}~,~ \mu(T)[1]).$$
    Here the sum is over all saturated essential combinatorial types of $P$. (Essential saturated types with $\cS_T \subseteq P$).
\end{thm}

To study pointed maps, we will use a slight strengthening \cref{VassilievSSDescription}, involving a generalization essential elements.

Let $R \subseteq W<\Hilb(C)^Q$ be a downwards closed union of strata $\cS_T$ for $T$ a saturated combinatorial type, which is initial in its upwards closure in the following sense:
\begin{itemize}
    \item if $w < y \in W<Q^{\rJ C}$ is in the upward closure of $R \cap W<Q^{\rJ C}$, then there exists $\eta_R(w < x) \in Q^{\rJ C}$ such that $w < \eta_R(w<y) \in R$ is the maximal element of $R \cap Q^{\rJ C}$ with $\eta_R(w<x) \leq y$.
\end{itemize}

\begin{ex} We may take the subposet $R \subseteq  W< \Hilb(C)^Q$ to be the collection of elements $w< x$ such that $w<\sat(x)$ is essential.  This is the main example of $R$, but we need slightly more flexibility to discuss spaces of pointed maps below.
\end{ex}

It follows formally that $R \cap w < Q^{\rJ C} \subseteq w < Q^{\rJ C}$ is closed under joins in $w < Q^{\rJ C}$, hence $R \cap w < \Hilb(C)^Q$ is closed under joins as well. 
(The join in $w< Q^{\rJ C}$ of two elements $w<x_1, w< x_2$ equals $w < \sat(x_1 \vee x_2)$).   For element $w<x$ in the upwards closure of $R$ with $x$ not necessarily saturated we extend the definition of $\eta_R$ by $\eta_R(w< x) := r(\sat(w<x))$.

We say that a saturated type $T$ is a \emph{type of $R$} or write $T \in R$ (resp.\ $P$) if $\cS_T \subseteq R$ (resp.\ $\cS_T \subseteq P$).   The same argument as in the case where $R$ consists of essential elements, immediately yields the following.

\begin{thm}\label{VassilievSSDescription2}
    If $P$ is contained in the upwards closure of $R$, there is a spectral sequence converging to  $H^*_c(|\rB(U, P, Z)|)$ with total $E_1$ page $$E_{1}^{*,*} = \bigoplus_{T \text{ type of $R \cap P$ }} H^*_c(Z|_{U \times_W \cN_T}~,~ \mu(T)[1]).$$
    Here the sum is over all (saturated) types of $R \cap P$.
\end{thm}

\subsection{Gysin Maps}

We assume that $\cB, U, E, P, R$ are as in the previous subsections, and that $P$ is contained in the upwards closure of $R$.  We also assume in  that $E$  is a finite dimensional real vector bundle over $U$ of constant rank. We have the following theorem on Gysin maps, when the subspace $Z \subseteq E \times_U P$ is linear.

\begin{thm}\label{GysinIso}
	Suppose that $Z_{b,w < x} \subseteq E_{b,w}$ is a linear subspace for all $(b,w) \in U, w< x \in P$.   Let $E' \subseteq E$ be an inclusion of vector bundles of constant rank over $U$, and let $Z'\to Z$ be a compatible map of stratifications by $P$.   Suppose that for every element $p \in P \cap R$  that $Z'_p \to Z_p \times_E E'$ is an isomorphism.   Then there is an induced Gysin isomorphism	
    $$H^{*- \dim(E/E')}_c( \rB(U, P, Z')) \to H^{*}_c( \rB(U, P,Z)).$$
\end{thm}
\begin{proof}
	Write $i^*Z$ for the pullback representation of $P$ defined by $(i^*Z)_p = Z_p \times_{E} E'$ (the notation is justified by setting $i$ to be the inclusion $E' \subseteq E$).   There is a factorization $Z' \to i^*Z \to Z$.   The pullback square induces a Gysin map $$H^{*- \dim(E/E')}_c(\rB(U, P, i^*Z)) \to H^{*}_c( \rB(U, P,Z)),$$ and there is an ordinary pullback map $$H^{*}_c( \rB(U, P, i^*Z)) \to H^{*}_c( \rB(U, P,Z')).$$ 
    Both these maps are isomorphisms by stratifying as in \cref{stratifyingbarconstructions} and applying \cref{VassilievSSDescription2}. The functoriality of the spectral sequence under Gysin maps frollows from the discussion of \S\ref{subsec:Gysin}.
\end{proof}

\subsection{Dimensions in semialgebraic geometry}
For the next subsection, we need to invoke a notion of dimensions that behaves well with respect to closure and compactly supported cohomology.  Thus we will work with semialgebraic spaces.
 A \emph{semialgebraic subset} of $\bbR^n$ is a set that can be obtained as a finite boolean combination of sets cut out by finitely many polynomial equalities an inequalities.  We summarize several key results about dimensions of semialgebraic sets.  (All of these results hold in the more general setting of definable sets with respect to an o-minimal structure, replacing the adjective ``semialgebraic" by ``definable"). 

Recall from \cite[Chapter 8]{tametopology} that a \emph{semialgebraic triangulation} is a semialgebraic homeomeorphism to a \emph{complex} $J \subseteq \bbR^n$, which by definition is a finite disjoint union of relatively open linear simplices $\sigma_i \subseteq \bbR^n$ satisfying $\overline{\sigma_i} \cap  \overline \sigma_j$ is empty or equal to $\overline \tau$ for $\tau$ a common face of $\sigma_i,\sigma_j$ (not necessarily in $J$).  Given a complex, consider the set $Z_i:= \overline \sigma_i \cap J$ with the partial order induced by containment.  Then $Z_i$ is a stratification satisfying the condition that $x \in J$ is contained in $\sigma_\alpha$ if and only if $Z_\alpha$ is the minimum of all the $Z_i$ containing $x$.  This implies the stratification satisfies the maximality condition from \S\ref{subsec:Gysin} and also that $\sigma_\alpha = Z_\alpha - \cup_{ Z_\beta \subsetneq Z_\alpha} Z_\beta$.  Therefore, if $J$ is locally compact, we obtain a compactly supported cohomology spectral sequence $\bigoplus_{\alpha} H^\bdot_c(\sigma_\alpha) \implies H^\bdot_c(J)$, which we will use below.

\begin{thm}\label{ominsummary}
    Every semialgebraic set $X \subseteq \bbR^n$ has a numerical dimension $\dim X \in \bbN$ (invariant under semialgebraic bijections) satisfying the following properties.
    \begin{enumerate}
        \item The dimension of a semialgebraic open subset of $\bbR^d$ is $d$, and the dimension of a (nonempty) finite union is the maximum of the dimension of the pieces.
        \item   The closure of a semialgebraic set $X$ is semialgebraic and $\dim \overline X = \dim X$.
        \item The image of a semialgebraic set $X$ under a semialgebraic map $f$ is semialgebraic and $\dim f(X) \leq \dim(X)$
        \item If $f: X \to Y$ is a semialgebraic map with finite fibers, then $\dim X \leq \dim Y$.  More generally, if $f: X \to Y$ is a semialgebraic map, then $\dim X = \max_{d \in \bbN} \dim(S(d)) + d$, where $S(d) \subseteq Y$ is the locus where the fiber has dimension at most $d$. (The set $S(d)$ is automatically semialgebraic).

        \item If $\Theta \subseteq X$ is a semialgebraic subset and $X$, then there is a semialgebraic triangulation of $X$ that is compatible with $\Theta$ (in the sense that $\Theta$ is a union of simplices).

        \item  If $X$ is locally compact, then $H_c^i(X) = 0$ for $i > \dim X$. 
    \end{enumerate}
\end{thm}
\begin{proof}
(1) follows from the definition of dimension and \cite[Thm 4.1.3]{tametopology}, (2) follows from \cite[Thm 4.1.8]{tametopology},  (3) and (4) follow from \cite[Cor 4.1.6]{tametopology}, and (5) is \cite[Thm 8.2.9]{tametopology}. Finally to prove (6), choose a triangulation of $X$ by (5). The dimension of any simplex appearing is $\leq \dim X$ and so the compactly supported cohomology spectral sequence associated to the triangulation (discussed above) gives that $H^{i}_c(X) = 0$ for all $i > \dim X$.
\end{proof}

A \emph{semialgebraic space} $Y$ is a topological space, endowed with a finite open cover $U_i$ and homeomorphisms $h_i: U_i \to U_i' \subseteq \bbR^n$ where $U_i'$ is semialgebraic (though not necessarily open) and such that the subsets $U'_{ij}:=U_i' \cap h_i(U_j)$ are semialgebraic and the homeomorphisms $h_j h_{i}\inv : U'_{ij} \to U'_{ji}$ are semialgebraic.  This allows one to extend the definition of semialgebraic subsets and semialgebraic maps to $Y$.  A \emph{semialgebraic morphism} $f: X \to Y$  is a continuous map whose graph is semialgebraic.  We may extend the previous results to semialgebraic spaces by the following.

\begin{thm}[{\cite[Theorem 10.1.8]{tametopology}}]
    A regular $(T3)$ semialgebraic space $Y$ is semialgebraically isomorphic to a semialgebraic subset of $\bbR^n$.
\end{thm}

Recall that a subspace of a regular space is regular, so this condition is necessary.  Moreover a locally compact Hausdorff space is regular.

\subsection{Approximation criterion} \label{subsec:approximationcriteria}

We take the notation and assumptions on $U$, $\cB$, $E$, $Z$, $R$ from \S\ref{stratifyingbarconstructions}.  We further assume that $\cB$ is a locally compact Hausdorff semialgebraic space, that $U$ is a semialgebraic subspace, $E \to U$ is a semialgebraic real vector bundle of constant rank, and $Z \subseteq E \times_U (U \leq \Hilb(C)^Q)$ is semialgebraic when restricted to $\Hilb_{\leq N}(C)^Q$ for all $N$.  (It is possible to weaken the assumption that $E$ is a vector bundle, but for simplicity we will not do so.  As in the previous subsection we may generalize immediately from semialgebraic space to ones definable with respect to an o-minimal structure). 

We further assume that
\begin{itemize}
    \item the natural containment $Z_{b,w \leq x} \subseteq Z_{b, w \leq \sat(x)}$ is an equality;
    \item if $w \leq x_1$ and $w \leq x_2$ are two elements then $Z_{b, w \leq x_1} \cap Z_{b, w \leq x_2} = Z_{b, w \leq x_1 \vee x_2}$.  
\end{itemize}  

We suppose that we are given an \emph{expected codimension function} $\gamma: W \leq \Hilb(C)^Q \to \bbN$ that only depends on the combinatorial type.  We say that $Z_{b, w \leq x}$ is \emph{weakly unobstructed} if its codimension in $E_{b,w}$ is $\geq \gamma(w \leq x)$ and is \emph{unobstructed} if equality holds.  (In the examples of interest, $\gamma$ will be chosen so that weak unobstructedness implies unobstructedness).

 Using $\gamma$, we define a function $\kappa$ on relative combinatorial types by $$\kappa(T) := \gamma(T)  
 - \rank(T) - 2|\supp(T)|.$$  Here $\rank(T)$ is defined as in \cref{rankex} and $|\supp(T)|$ is defined as in \cref{suppex}. (Note that for these functions $f(w\leq x)$ only depends on the combinatorial type of $w\leq x$, so that $f(T)$  is well defined). Our convention in the pointed and relative case is that, if $T = (g_0 \geq h_0, \{g_1 \geq h_1, \dots, g_r \geq h_r\})$ then $|\supp(T)|$ is the cardinality of $\{i > 0 ~|~ g_i > h_i\}$.  
 
 Heuristically, $\kappa$ records the expected codimension of the homological contribution of the stratum $\cS_T$ to the compactly supported cohomology of bar construction. 
 In other words we expect $\cS_T$ to only affect compactly supported cohomology in degrees $\leq \dim(E) - \kappa(T)$.  The following criterion relating the compactly supported cohomology of the bar construction to the compactly supported cohomology of $\im(Z|_P) \to E$ substantiates this heuristic.  (In the statement $P \subseteq W< \Hilb(C)^Q$ is a poset with the properties assumed in the beginning of this section, and $R \subseteq W< \Hilb(C)^Q$ satisfying the properties assumed in \S\ref{stratifyingbarconstructions}).

\begin{thm}\label{BarApproximation}
 	Let $I \in \bbN$.   Suppose that the following conditions hold. 
		\begin{enumerate}
            \item $P \to W$ is proper, $P$ is closed and downward closed.
			\item For every pair $(b, w < x)$ in $P_U$ such that $x$ is saturated, and every  $y \in Q^{\rJ C}$ such that $x \prec_{\rm sat} y$ and $(w < y) \in R - P$, the fiber $Z_{b,w <y}$ is weakly unobstructed. (Here and throughout $x \prec_{\rm sat} y$ means  $x \leq y$ and there is no saturated element strictly between $x$ and $y$, given $x$ and $y$ saturated).
			\item $P$ is contained in the upwards closure of $R$, the support of $\eta_R(w<x)$ equals the support of $x$ for $(w< x) \in P$, and $P$ contains every $\cS_T$ such that $T$ is a type of $R$ with $\kappa(T) \leq I$.
		\end{enumerate}
Then the map induced on compactly supported cochains  by $f: B(P,Z) \to \im(Z|_P)$is connected in codimension $I+2$.    
(In other words, for all $i > \dim E - I - 2$ the map $H^i_c(B(P,Z)) \leftarrow H^i_c(\im(Z|_P))$ is an isomorphism and for $i = \dim E - I - 2$ it is a surjection).
 \end{thm}

\begin{remark}
    It is possible to prove \cref{BarApproximation} without the assumptions on semialgebraicity on $Z$, $\cB$, $E$, $U$, provided that $R = W< \Hilb(C)^Q$, and $Z \subseteq E$ is a linear subspace. (A previous version of this paper contained a statement without semialgebraicity hypotheses, but the proof was incomplete without this restriction on $R$).
\end{remark}

\begin{proof}
We write $\pi$ for the map $Z \to E$ (not necessarily restricted to $P$).  For every saturated type $T$ of $R$ we define  $J_T = \pi(Z|_{\cN_T})$.  Since it is an image of an intersection of semialgebraic sets, we have that $J_T$ is semialgebraic.

Let $$J_T^\circ:= J_T \setminus \bigcup_{T < S \text{ a type of $R$}} J_S.$$

If $T < S$ then $T \prec_\sat T' \le S$ for some $T'$ (and if $S \in R$ then so is $T'$), so 
\[J_T^\circ = J_T \setminus \bigcup_{T \prec_\sat T' \text{ a type of } R} J_{T'}.\]
Since 
there are only finitely many $T' \succ_\sat T$ for a given $T$, we have that $J_T^\circ$ is semialgebraic.  We put $K := \pi(Z|_P)$ and $\Theta = K - \bigcup_{T \in P \cap R} J_T^\circ$, where the union ranges over all saturated types of $P \cap R$.  The main facts about $\Theta$ and $J_T^\circ$ that we will use are contained in the following lemma.

\begin{lem}\label{increase} We have:

\begin{enumerate}
   \item  Let $T$ be a saturated combinatorial type of $R$.  Suppose $s \in Z_{b, w<x}$ for $x \in \cN_T$. If $(w,b,s) \notin J_T^\circ$  then there exists a saturated $x' \succ_{\rm sat} x$  such that $(w<x') \in R$ and $s \in Z_{b, w<x'}$.    Conversely, if $(b,w,s) \in J_T^\circ$, then $(w<x)$ is a maximum of all saturated elements of $y \in R \cap (w< Q^{\rJ C})$ satisfying $s \in Z_{b, w < y}$.
     
\item  If $(b,w,s) \in \Theta$ then  there exists a saturated  $(w<y) \in R - P$ such that $y' \prec_{\sat} y$ and $(w<y') \in P.$    Moreover, if $s \in Z_{b, w<x}$ for $(w<x) \in P \cap R$ saturated then we may choose $y$, $y'$ so that $y > y' \geq x$.  

\end{enumerate}
\end{lem}
\begin{proof}
For (1), suppose $(b,w,s) \notin J_T^{\circ}.$  Then there exists $T' > T$ a saturated type of $R$ such that $s \in Z_{b, w<y}$ for some $w<y$ of type $T'$.  So $s \in  Z_{b,w<y} \cap Z_{b,w<x}  = Z_{b, w<y \vee x}= Z_{b, w <\sat(y \vee x)}$.  Because $R \cap w<Q^{\rJ C}$ is closed under saturated meets, we have $\sat(y \vee x) \in R$ and $\sat(y \vee x) \geq x$ because $x$ is saturated.  Hence we may refine this chain to get $x'$ satisfying the desired condition.

Conversely, suppose $(b,w,s) \in J_T^\circ$ and $y \in R \cap (w< Q^{\rJ C})$ satisfies $s \in Z_{b, w < y}$.  Then as above we have $\sat(y \vee x) \in R$,  $\sat(y \vee x) \geq x$, and $s \in  Z_{b, w <\sat(y \vee x)}$. Hence $\sat(y \vee x) = x$ and so $y \vee x = x$ implies $y \leq x$.

For (2), suppose $(b,w,s) \in \Theta$. Then because $s \in K$  we have $s \in Z_{b, w<a}$ for some $w<a \in P$, so $s \in Z_{x_0}$ for $x_0 = \eta_R(\sat(a))$. Then because $x_0$ is saturated and in $R \cap P$, and we know by assumption that $s \notin J^\circ_{\type(x_0)}$, we know by part (1) that there is a saturated element $x_1 \succ_{\sat} x_0$ in $R$ such that $s \in Z_{b, w<x_1}$.  If $(w< x_1) \notin P$ then we are done (with $y = x_1, y' = x_0$), otherwise we may continue inductively to find the desired $(w<y)$.  If we are given a saturated $(w< x) \in P \cap R$ with $s \in Z_{b, w<x}$, we may repeat the same argument with $x_0 = x$. 
\end{proof}

 By \cref{ominsummary} (5), we choose a triangulation of $K$ that is compatible with $\Theta$ in the sense that $\Theta$ is a union of simplices. Note that $K$ is locally compact since $P$ and hence $Z|_P$ is closed and $\pi$ is proper. Then by the compactly supported cohomology spectral sequence associated to the triangulation it suffices to prove that $H^i_c(f\inv(\sigma)) \leftarrow  H^i_c(\sigma) $ is an isomorphism for $i > \dim E - I - 2$ and a surjection for $i = \dim E - I - 2$, for any simplex $\sigma$ of the triangulation.  By compatibility of the triangulation with $\Theta$ we have two cases: (I)  $\sigma \subseteq K - \Theta$ and (II) $\sigma \subseteq \Theta$.

In case (I), fix $(b,w,s) \in \sigma$. We will prove that  $f\inv((b,w,s))$ is contractible for all $(b,w,s) \in \sigma$ (this suffices by Leray).  We have that $(b,w,s) \in J_T^\circ$ for some $T$ a saturated type of $P \cap R$. By definition,  there exists $x_s \in \cN_T$ such that $s \in Z_{b, w < x_s}$, and by \cref{increase} we have that $x_s$ is a maximum of all saturated elements $x$ in $R \cap(w< Q^{\rJ C})$ satisfying $s \in Z_{b, w<x}$.   

Consider $\cV:= \{ w< x \in P ~|~ s \in Z_{b,w<x}\}$.  If $x \in \cV$ then $\eta_R(w<x) \leq x$ so $\eta_R(w<x) \in \cV$  and so $\eta_R(w<x) \leq x_s$ because $x_s$ is maximal.  Therefore $\cV \cap R$ has a maximal element $w < x_s$.  We further claim that $\cV$ is finite.  Because there are finitely many elements $\leq x_s$, it suffices to prove that $\eta_R \inv(y) \cap P$ is finite for all $y \in Q^{\rJ C}$.  This follows because $P$ contains finitely many combinatorial types,  we have assumed that $\eta_R$ preserves support, and for any $y$ the number of elements of a fixed type with the same support as $y$ is finite.

The fiber $f\inv(b,w,s)$ is the nerve of the finite poset $\cV$.  By the adjoint $\eta_R$ to the inclusion of $\cV \cap R \to \cV$, this nerve is equivalent to the nerve of $\cV \cap R$ which is contractible because $(w < x_s)$ is a maximal element.

For case (II), we suppose $\sigma \subseteq \Theta$. We will  prove that $H^i_c(f\inv(\sigma)) = H^i(\sigma) = 0$ for $i \geq \dim E - I - 2$.  We first show this for $f\inv(\sigma)$.
    We stratify $f\inv(\sigma)$, which is the geometric realization of the semi-simplicial space $\rB(U,P, \sigma \cap Z)$ defined by $$[r] \mapsto  \{(b,w) \in U, w  < x_0 < \dots < x_r \in P,  s \in \sigma \cap Z_{b,w<x_r}\},$$ using saturated types of $P$ as in  \S\ref{stratifyingbarconstructions}.  Given a saturated type $T$ of $P \cap R$, we let $Y_{T}$ denote the semi-simplicial space with $(Y_T)_r =$
    \[  \{(b,w) \in U, w < x_0 < \dots < x_r \in Q^{\rJ C} \cap P ,  s \in \sigma \cap Z_{b,w<x_r} ~|~  w< x_r \in \cN_T  \}. \]  From the spectral sequence \cref{VassilievSSDescription2}, associated to the stratification of $\rB(U,P, \sigma \cap Z)$ it suffices to establish that  $\dim (Y_{T})_r + r < \dim E - I - 2$. (Note that $(Y_T)_r$ is locally compact, because it is constructed in $\S\ref{stratifyingbarconstructions}$ as an open subset of a closed subset of locally compact spaces).  

    We write $\suc_R(P)$ for the set of saturated elements $w < y \in R - P$ such that there exists $y' \prec_{sat} y$  with $w < y' \in P$. (Note that $\suc_R(P)$ is semialgebraic, because $R, P$ and the relation $\leq$ are semialgebraic).     We let $Y_{T}'$ be the semi-simplicial space with $r$-simplices 
        \[\{(b,w) \in U, w < x_0 < \dots < x_r \in P \cap Q^{\rJ C},  y \in \suc_R(P),  s \in \sigma \cap Z_{b, w<y}~|~ x_r \in \cN_T, y > x\}.\]
    There is a map $(Y_T') \to (Y_T)$ given by forgetting $y$. By the second statement of \cref{increase}, this map is surjective. 
    We decompose  $(Y_{T})'_r $ according to type $T'$ of $w < y$, $(Y_{T}')_r = \cup_{T'}  (Y_{T,T'}')_r$.      Therefore by (2) and (1) of \cref{ominsummary}, we have that $ \dim ((Y_T)_r )\leq  \dim( (Y_{T} ')_r )\leq \max_{T' \in \suc_R(P)} \dim ((Y_{T,T'}')_r) $ 
        
    By assumption (2) we have that every $w<y \in \suc_R(P)$ is unobstructed.  Therefore the map taking an element of $(Y_{T,T'}')_r$ to $w < y \in \cN_{T'}$ has fibers given by a finite set times a set of dimension $\leq \gamma(T')$, hence the dimension of $(Y_{T,T'}')_r$ is less than or equal to $\dim E - \gamma(T') + 2|\supp(T')|$ by (4) of \cref{ominsummary}. And if $(Y_{T,T'}')_r \neq \emptyset$  we have $r \leq \rank(T') - 2$, which implies  $$\dim (Y_{T,T'})_r + r \leq \dim E - \kappa(T') - 2 < \dim E - I - 2 ,$$  where in the last step we have used that $T'$ is not a type of $P$ and hypothesis (3), which imply $\kappa(T') > I$.  Combining the inequalities we obtain the desired bound on $\dim ((Y_T)_r)$.
    
    Finally, a similar application of \cref{increase} shows that $\dim(\sigma) \leq \dim E - I - 2$ and so we are done.  
\end{proof}

\section{Semi-topological model and finite dimensional approximations} \label{sec:SemiTop}

            To compare algebraic maps $C \to {\rm Bl}_{p_1, \dots, p_r} \bbP^n$ with continuous ones, we pass through an intermediate model, which we call a semi-topological model.

\subsection{Continuous sections vanishing with multiplicity}

\begin{defn}
Let $U \subseteq C$ be an open subset and let $w: K \to \Sym^n U$ be a family of divisors of $C$, where $K$ is a compact set.  We say that a family of functions $f_k: U \to \bbC, k \in K$  \emph{vanishes to order $\geq w$} if
    \begin{itemize}
	\item for every $u \in U, k \in K$ there is an $a_{u,k} \in \bbC$ such that  $f_k = a_{u,k} \rho_{w(k)} + o(|\rho_{w(k)}|)$ at $u$ uniformly in $K$.   (Such an $a_{u,k}$ is necessarily unique)
\end{itemize}
If $a_{u,k} \neq 0$ for all $u,k$ we say that the family vanishes to order \emph{exactly $w$}.   Finally, if $f_k$ equals $q_k + g_k$ where $q_k$ is a continuous family of degree $n-1$ polynomials and $g_k$ vanishes to order $\geq w$, then we say that $f_k$ is a \emph{$w$-holomorphic family}.
\end{defn}

\begin{prop}\label{localrepresentability}
    Let $w:K \to \Sym^n(U)$ be given.  Then:
    \begin{enumerate}
        \item Multiplication by $\rho_{w(k)}$ induces a bijection   
    \[\{ g: K \times U \to \bbC \text{ continuous }\} \to \{f_k \text{ family vanishing to order $\geq w$} \}\] 
        \item Given a biholomorphism $h : U' \to U$, precomposition with $h$ induces a bijection \[\{f_k: U \to \bbC  \text{ vanishing to order $\geq w$}\}  \to  \{f'_k: U' \to \bbC  \text{ vanishing to order $\geq h^*(w)$} \}\]
    \end{enumerate}
\end{prop}
\begin{proof}
    A function $K \times U \to \bbC$ denoted $g_k(u)$ is continuous if and only if for every $u \in U$ there exists a constant $a_{u,y}$ (which must equal $g_k(u)$) such that  $g_k = a_{u,k} + o(1)$  at $u$.  Hence if $g_k$ is continuous $\rho_{w(k)} g_k$ vanishes to order $\geq w$, and multiplication by $\rho_{w(k)}$ defines an injection from the RHS to the LHS. Conversely, if $f_k$ satisfies (*) then 
    \[g_k(u) := \begin{cases*}
        f_k(u)/\rho_{w(k)} & if $\rho_{w(k)} \ne 0$\\
        a_{u, k} & if $\rho_{w(k)} = 0$
        \end{cases*}
    \]
    is continuous by the same criterion.
   
        To obtain (2) from (1), we write $c_k = \rho_{w(k),U} \circ h$ and $b_k = \rho_{h^*w(k), U'}$ and use that   \[\Top(U', \bbC) \times K \xrightarrow{(c_k/b_k) \cdot} \Top(U', \bbC) \times K\] is a homeomorphism because $c_k/b_k$ does not vanish on $U'$.
\end{proof}

\begin{cor}
    Consider the space of pairs $$X:= \{(w \in \Sym^n(U), f: U \to \bbC \text{ vanishing to order }  \geq w \}.$$ There is a unique compactly generated topology on $X$ such that continuous maps $f: K \to X$ for $K$ compact correspond  to families of functions indexed by $K$ vanishing to order $\geq w$.  Furthermore, this topology is invariant under biholomorphisms $h: U \to U'$.
\end{cor}
\begin{proof}
    Endow $X$ with the topology induced by the bijection $\Sym^n(U) \times \Top(U,\bbC) \to X$ given by multiplication by $\rho_w$, and use \cref{localrepresentability}. 
\end{proof}

We also have a variant of this corollary for $w$-holomorphic maps.
 
\begin{prop}
   There is a unique compactly generated topology on the set $$Y = \{(w \in \Sym^n(U), f: U \to \bbC \text{ $w$-holomorphic}) \}$$  such that continuous maps $f: K \to X$ for $K$ compact correspond  to $w$-holomorphic families of functions indexed by $K$.  This topology is invariant under biholomorphisms $h: U \to U'$.
\end{prop}
\begin{proof}
    There is an bijection $Y \to X \times {\rm Poly}_n(\bbC)$, where $$X:= \{ (w \in \Sym^n(U), f: U \to \bbC \text{ vanishing to order }  \geq w) \},$$ because by definition any $w$-holomorphic $f$ can be written as a sum of a degree $n$ polynomial $q$ and a function $g$ vanishing to order $w$.  (The Taylor expansion of $f$ at the points of $w$ determines $q$ uniquely).  We give $Y$ the product topology.
    
    To prove invariance, suppose that $U$ has two holomorphic coordinates $z, z'$. Let $f_k$ be a family of $w$-holomorphic functions with respect to $z$, and $q_k,q_k'$ be the family of degree $n-1$ polynomials in $z,z'$ respectively whose Taylor expansion agrees with $f_k$ at $w(k)$.  By definition, we have that $f_k - q_k$ vanishes to order $\geq w$.  And $f_k -q_k' = f_k - q_k + (q_k - q_k')$.  Since $q_k - q_k'$ is the difference of two holomorphic functions with the same residues at $w(k)$ it vanishes to order $\geq w$, so $f_k - q_k'$ must as well.
\end{proof}

Now let $C$ be a Riemann surface, and let $w: K \to \Sym^n C$ be a family of divisors. For every $k_0 \in K$ there is 
\begin{itemize}
    \item  a neighborhood of $w(k_0)$ of the form $\prod_{i =1}^t \Sym^{n_i} U_i$ where the $U_i$ are disjoint open neighborhoods that are biholomorphic open subsets of $C$ and $n = n_1 + \dots + n_t$
    \item a neighborhood $V \ni k_0$ such that $w|_V$ factors through $\prod_{i =1}^t \Sym^{n_i} U_i$  as $\prod_{i = 1}^t w_i$
\end{itemize}  
We say that $f : K \times C \to \bbC$ is a family of functions vanishing to order $\geq w$ (resp. exactly $w$) if for every $k_0 \in K$ and every $i$ we have that $f|_{V \times U_i}$ vanishes to order $\geq w_i$ (resp. exactly $w_i$) for some (equivalently any) choice of $\prod_{i = 1}^{t}\Sym^{n_i}(U_i)$ and $V$ as above.

\begin{defn}\label{vanishingfamily} Let $M$ be a complex manifold and let $N \subseteq M$ be a complex submanifold.  Let $C$  be a Riemann surface, let $w: K \to \Sym^n C$ be a family of divisors, and let $f: K \times C \to M$ be a continuous map such that $f(k,p) \in N$ for every $p$ in the support of $w(k)$.  Then for every $k_0 \in K$  we have the following:
\begin{itemize}
    \item for every $p$ in the support of $w(k_0)$ there exist holomorphic local coordinates $s_1^p, \dots, s_{\dim M - \dim N}^p$ on a neighborhood $\cU_p \ni f_{k_0}(p)$, whose vanishing locus cuts out $N \cap \cU_p$
    \item there are disjoint closed balls $B_p \ni p$ contained in $f\inv(\cU_p)$ and an open neighborhood of $w(k_0)$ of the form $\prod_{p} \Sym^{n_p} U_p$ whose support is contained in the union of the balls $B_p$
    \item  there is a neighborhood $V$ of $k_0$ such that for every $k \in V$ we have that $f(B_p) \subseteq \cU_p$ and $w_V$ factors through $\prod_{p} \Sym^{n_p} U_p$ as $\prod_{p} w_p$
\end{itemize}  
We say that $f: K \times C \to M$ intersects $N$ to order $\geq w$ if for some (equivalently every) choice of $\cU_p, B_p, U_p, V$ as above we have that $s_i^p|_{V \times U_p}$ vanishes to order $\geq w_p$ for all $p, i$.   We say that it intersects \emph{to order exactly $w$} if $f_k\inv(N)$ equals the support of $w_k$ and for every $p$ there is at least one $i$ such that $s_i^p|_{V \times U_p}$ vanishes to order exactly $w_p$. 
\end{defn}

\begin{remark}
	Note that in the above definition, the holomorphicity conditions are only imposed on the coordinates whose vanishing defines $N$.
\end{remark}

\begin{defn}
    For a single function $f : U \to \bbC$ and $w \in \Sym^n(U)$, we can define each of the above notions by taking $K = \text{pt}$.
    If further $w = nu$ for some $u \in U$ then we will write $f$ vanishes to order $\ge n$ at $u$, $f$ vanishes to order exactly $n$ at $u$, and so on. 
\end{defn}

\begin{remark}
    Note that if $f$ vanishes to order $\ge w$ but does not vanish to order exactly $w$ (i.e.\ $a_{u, k}$ vanishes somewhere), it is not necessary that $f$ vanishes to order $\ge w'$ for some $w' > w$, even if $K = \text{pt}$ and $w = nc$ for some $c \in U$. There just might not be a polynomial that approximates $f$ to the required precision.
\end{remark}

\begin{prop}\label{twisted-section-universal-property-topology}
    Let $N_1, \dots, N_r \subseteq M$ be complex submanifolds of $M$.   Let $n_1, \dots, n_r \in \bbN$.  There is unique compactly generated topology on the set  
    \[T = \left\{(D_i)_{i  = 1}^r \in  \prod_{i= 1}^r \Sym^{n_i} C, f: C \to M~|~  f \text{ intersects $N_i$ to order $\geq D_i$} \forall i \right \}\] such that for all compact $K$ continuous maps $f: K \to T$ correspond to the data of a continuous map $f: K \times C \to M$ and maps $w_i: K \to \Sym^{n_i} C$ for every $i$ such that $f_k$ is a family of maps intersecting $N_i$ to order $\geq w_i$.
\end{prop}
\begin{proof}
    By taking fiber products, we reduce to the case where $r = 1$.  Then we may apply the following general statement about the relation between local representability and global representability \cref{abstractextension} in the case where  
    \begin{align*}X &= \{ w \in \Sym^n C, f: C \to M ~|~ h \text{ intersects $N$ to order $\geq w$ }\} \\
    &\subseteq \Sym^n C \times \Top(C,M),
    \end{align*}
    with $\tilde U$ the basis of open subsets associated to choices of  $w_0 \in \Sym^n C$ and  $(\cU_p, B_p, U_p)_{p \in {\supp(w_0)}}$ as in \cref{vanishingfamily}.   (Here the open subset $V$ of $X$ associated to such a datum is the set of all pairs $(w,f) \in X$ such that $w \in \prod_{p} \Sym^{n_p} U_p$ and $f(B_p) \subseteq U_p$.  And $\cF(K,U) \subseteq \Top(K,U)$ is the subsheaf of pairs $w: K \to \prod_{p}\Sym^{n_p} U_p$ and $f:K \times C \to M$ where $f$ is a $w$-holomorphic family).
\end{proof}
    
\begin{lem}\label{abstractextension}    
    Let $X$ be a topological space.  Suppose that $X$ has a basis of open subsets $\tilde \cU$ and for every $U \in \tilde \cU$ there is a subsheaf $\cF(-,U) \subseteq \Top(-,U)$ on the category of compact Hausdorff spaces (with Grothendieck topology given by finite collections jointly surjective maps)  such that for all $U \subseteq V \in \tilde \cU$ we have that $\cF(K,U) = \cF(K,V) \cap \Top(K,U)$. 

    Then there is a unique extension of $\cF(K,W)$ to all open subsets $W \subseteq X$ such that $\cF$ is a sheaf in $K$ and for all $W_1 \subseteq W_2$ we have $\cF(K,W_1) = \cF(K, W_1) \cap \Top(K,W_2)$.  Moreover, if $\cF(-,U)$ is representable by a compactly generated topology on $U$ for all $U \in \cU$ then $\cF(-,W)$ is representable by a compactly generated topology on $W$.
    
\end{lem}
\begin{proof}
   Let $f \in \Top(K,W)$.  Writing $W = \cup_{\alpha} U_\alpha$ for $U_\alpha \in \tilde \cU$ we have that there is a finite cover of $K$ by compact neighborhooods $K_1, \dots, K_N$ such that $f|_{K_i} \in \Top(K_i, U_{\alpha(i)})$ for some choice of $\alpha(i)$. (Take the open cover $f\inv(U_\alpha)$ of $K$ and find a subcover of open neighborhoods $V_{\alpha,\beta} \subseteq f\inv(U_\alpha)$ such that  $\overline V_{\alpha,\beta} \subseteq f\inv(U_\alpha)$.  By compactness, we may choose a finite subcover, $V_1, \dots, V_N$ and take $K_i = \overline V_i$).  We declare that $f \in \cF(K,W)$ if each $f|_{K_i}$ lies in $\cF(K_i, U_{\alpha(i)})$.   This definition is independent of the choice of covers: for any other choice of $K_j, U_{\alpha(j)}$ such that $f|_{K_j} \in \Top(K_j, U_{\alpha(j)})$ there is a common refinement by $K_i \cap K_j, U_{\alpha(i)} \cap U_{\alpha(j)}$ and the sheaf property implies that $f|_{K_j}$ lies in $\cF$ if and only if $f|_{K_i \cap K_j}$  does for all $j$.  By definition $\cF(-,W) \subseteq \Top(-,W)$ is the unique subsheaf on compact Hausdorff spaces such that $\cF(-,W) \cap \Top(-,W) = \cF(-,U)$ for all $U \in \tilde \cU$.
   Moreover, we have that $f \in \cF(K,W)$ if and only if for every compact $J \subseteq K$ such that $f|_J$ factors through $U \in  \tilde \cU$, the map $f|_J$ lies in $\cF(J,U)$.  

   To construct a topology on $W$ representing $\cF(-,W)$, we declare a subset $L \subseteq W$ to be open if and only if it is open for every intersection $U \cap L$ with $U \in \tilde \cU$.  Then $f: K \to W$ is continuous if and only if $f|_{f\inv(U)} : f\inv(U) \to U$ is continuous for all $U \in \tilde \cU$ if and only if $f|_K \in \cF(K,U)$ for all compact $K \subseteq f \inv(U)$.
 \end{proof}

\subsubsection{Case of interest}

We now consider the case where $M = V$ and $N_i = \ell_i \subseteq V$. 

Let $W_n \subseteq  \prod_{i = 1}^r \Sym^{n_i} C$ be the open subset of the product consisting of divisors with pairwise disjoint support.   Let $(D_i)_{i = 1}^r$ be the universal family of pairwise disjoint divisors, so that the fiber of $D_i$ above $w \in W_n$ is $w_i$.
Define $E^\cts = $
\[\{(w \in W_n, L \in \Pic^d C ,~ \text{section $C \to L \otimes V$ intersecting $L \otimes \ell_i$ to order  $\geq w_i$  $\forall i$})\},\] 
topologized as in \cref{twisted-section-universal-property-topology}, and let $E_{w, L}$ be its fiber over $(w, L) \in W_n \times \Pic^d C$.  
Similarly, let $E^\alg$ be the subset of $E^\cts$ consisting of algebraic sections, with fiber $E^\alg_{L, w}$ over $(L, w) \in \Pic^d \times W_n$.

We have ways of identifying $E^\cts$ with the space relative sections of a modified family of vector bundles, one using explicit transition functions, and one using algebraic geometry.  We first describe the construction using transition functions.

Fix $w' \in W_n$.  Let $W_n$ be a neighborhood of $w'$ of the form $\prod_{i =1}^t \Sym^{n_i} U_i$ where the $U_i$ are disjoint open neighborhoods of $C$ equipped with holomorphic local coordinates identifying them with open subsets of $\bbC$. Using these local coordinates, we have a continuous family of polynomials $\rho_{i} = \rho_{w_i}$ parameterized by $w \in W_n$. There is a vector bundle $V(-D)$ on $W_n \times C$, such that for every $w \in W_n$ the fiber of $V(-D)$ over $w$ is given by multiplying together the transition maps for $\ell_i \oplus \ell_i^\perp (-w_i)$ on $U_i$. If $\cL_d$ is the Poincar\'e bundle on $\Pic^d C$, we can further consider the vector bundle $\cL_d \otimes V(-D)$ on $U \times \Pic^d(C) \times C$.

Consider the projection map $\pi: W_n \times \Pic^d C \times C \to W_n \times \Pic^d C$.  We write $\Gamma(\pi, \cL \otimes V(-D))$ for the space of relative sections, consisting of pairs $(w, L) \in W_n \times \Pic^d C$ and $s \in \Gamma(C, L \otimes V(-D)|_w)$.

\begin{prop}\label{uglycompare}
The map from
$\Gamma( \pi, \cL \otimes V(-D))$ to \[\{(w \in W_n, L \in \Pic^d C, ~ \text{sections $C \to L \otimes V$ intersecting $L \otimes \ell_i$ to order $\geq w_i$ $\forall i$})\}\]
 given by multiplying by $\id_{\ell_i} \oplus \rho_i$ on each $U_i$ is a homeomorphism where the codomain is endowed with the above topology (and the domain is endowed with the compact generation of the compact-open topology).
\end{prop}
\begin{proof}
It suffices to show that for any compact $K$, and any choice of $w_{i}: K \to \Sym^{n_i}(U_i)$ the map   from $\Gamma(K \times C, V(-D))$ to \[ \{\text{families of sections $K \times C \to V$ intersecting $\ell_i$ to order  $\geq w_i$}\}\] given by multiplication by $\id_{\ell_i} \oplus \rho_{i, w_i(k)}$ is a bijection.

Given a vector bundle $E \to C$ with a trivialization $E \to V \times U$ on $U \subseteq C$, a subspace $V_0 \subseteq V$ and a divisor $A \in \Sym^k(B)$ where $B \subseteq U$ we may form the twist $E(-A, V_0)$ by gluing $E|_{W - B}$ to $(V_0 \oplus V_0^\perp(-A)) \times \overline{B}$, where $V_0^\perp$ is a complement of $V_0$ in $V$, keeping in mind that $V_0^\perp(-A)$ is trivial away from the support of $A$.

If we start with $\cL_d \otimes V$ and apply the above procedure iteratively with $V_0 = L \otimes \ell_i$, $A = w_i$, $B = W_{n,i}$ for each $i$ and $V_0 = V$, $A = w_0$ then we obtain the bundle $\cL_d \otimes V(-D)$. 
Following the correspondences of sections applied to families over compact $K$ establishes the claimed homeomorphism.
\end{proof}

Next we construct algebraic (or holomorphic) vector bundles $\cE_{L,w}$ such that there is a commutative diagram
\begin{center}
\begin{tikzcd}
{\Gamma^{\rm top}(C, \cE_{w,L})  } \arrow[r]           & E^\cts_{w L}           \\
{\Gamma^{\rm alg}(C, \cE_{w,L})  } \arrow[u] \arrow[r] & E^\alg_{w,L} \arrow[u]
\end{tikzcd}
\end{center}
 with horizontal arrows isomorphisms and vertical arrows inclusions.

Moreover, we construct $\cE_{w,L}$ so that as $w$ and $L$ vary, $\cE_{w,L}$ assemble into a bundle $\cE$ over $C \times \Pic^d C \times W_n$ and the horizontal isomorphisms in the diagram are induced by a morphism of bundles 
$\cE \to \cL_d \otimes V$.

We define the algebraic bundle $\cE$ and the morphism $\cE \to \cL_d \otimes V$ in terms of its sheaf of sections.  More precisely, the sheaf of sections (which we denote by $\tilde\cE$) is the fiber product of the following diagram of coherent sheaves on $C \times \Pic^d(C) \times W_n$, where $D_i \subseteq C \times W_n \times {\rm Pic}^d(C)$ denotes the  preimage of the universal subscheme from $C \times \Hilb_{n_i}(C)$:
\[\begin{tikzcd}
                                                 & \tilde \cL_d \otimes V \arrow[d]        \\
\prod_{i} \tilde \cL_d \otimes \ell_i |_{D_i} \arrow[r] & \prod_{i}  \tilde \cL_d \otimes V|_{D_i}
\end{tikzcd}\]
We claim that the fiber product $\tilde \cE$ is a locally free sheaf.  Indeed, $\tilde\cE$ is flat over $\Pic^d(C) \times W_n$ because each sheaf in the fiber product is and the vertical map is surjective.  So by an application of Nakayama's Lemma (see \stackscite{080Q}) it suffices to prove that $\tilde \cE_{w,L}$ is a vector bundle over $C$. But this holds because $C$ is a smooth curve, so every subsheaf of a locally free sheaf is locally free. 

We have a similar proposition to before.

\begin{prop}\label{vanishing-sections-vs-sections-of-twisted-bundle}
The map $\Gamma^\cts(\pi, \cE) \to \Gamma^\cts(\pi, \cL \otimes V)$ induces a homeomorphism $\Gamma^\cts(\pi, \cE) \to E^\cts$.
In particular, $E^\cts \to W_n \times \Pic^d C$ is a locally trivial bundle of separable Banach spaces.
\end{prop}
\begin{proof}
First note that $\Gamma^\cts(\pi, \cE)$ is locally trivial: it suffices to note that for a contractible neighborhood $U$ of $W_n \times \Pic^d$, the projection $U \times C \to C$ is a homotopy equivalence, so $\cE|_{U \times C}$ pulls back from a topological bundle on $C$.
Each fiber $\Gamma^\cts(\pi, \cE)_{(w, L)} = \Gamma^\cts(C, \cE_{w, L})$ is the space of sections of a vector bundle over $C$. Since $C$ is compact, the compact-open topology on it is given by a norm which makes it a separable Banach space. Further, since a convergent sequence along with its limit forms a compact set, any first-countable space is compactly generated.  Similarly by Proposition \ref{uglycompare}, $E^\cts$ is a locally trivial bundle of Banach spaces.  So it suffices to prove the map is a fiberwise homeomorphism. 

Accordingly, we fix $(w,L) \in W_n \times \Pic^d(C)$, and show that $$j : \Gamma^\cts(C, \cE_{w,L}) \to \Gamma^\cts(C, L \otimes V)$$ is injective with image equal to the continuous sections of $L \otimes V$ that intersect $L \otimes \ell_i$ to order $\geq w_i$ for all $i$. (This suffices by definition of $E^\cts$ and \cref{twisted-section-universal-property-topology}).  Injectivity follows from the fact that $\cE_{w,L} \to L \otimes V$ is an isomorphism away from the  support of $\cup_{i} w_i$.

Recall that the definition $\cE_{w,L}$ characterizes its algebraic sections.  To determine the image of $j$, we use that a continuous section is in the image of a map induced by a morphism of vector bundles if and only if it is so locally.  Hence we may replace $C$ by Zariski open neighborhoods $U \subseteq C$ on which $L$ admits an algebraic trivialization by a nonvanishing section $\phi$.  Further restricting $U$, we can assume that $U$ meets exactly one divisor $w_{i_0}$ and let $\rho$ be an algebraic function vanishing to order exactly $w_{i_0} \cap U$.  Choosing a basis of $V$ as $e_1,\dots, e_{\dim V}$  so that $e_1$ spans $\ell_{i_0}$, we have that $\phi e_1, \dots, \phi e_{\dim V}$ are an algebraic trivialization of $V \otimes L |_U$ and $\phi e_1, \phi \rho e_2, \dots,  \phi\rho e_{\dim V}$ are an algebraic trivialization of $\cE_{w,L}$, and $j|_U$ is the obvious inclusion.  
We have that $s\in E^\cts_{w,L}$ if and only if for all such choices of $U$,  writing $s|_U = \sum_{\beta = 1}^{\dim V} f_\beta \phi e_\beta$ we have that $f_2, \dotsc, f_{\dim V}$ are all divisible by $\rho$.  But by the description of the inclusion $\cE_{w,L}|_U\to L \otimes V|_U$  this condition precisely characterizes the image of $j|_U$.
\end{proof}

\begin{defn}\label{topological_stratification_definition}
    For $(w \le x) \in (W_n \le \Hilb(C)^Q)$ let $Z_{w \le x}$ be the subspace of $E_w$ consisting of those sections $s$ such that for all $c \in C$:
    \begin{itemize}
        \item if the multiplicity of $w_{\ell_i}$ at $c$ is $k$ and the multiplicity of $x_{\ell_i}$ at $c$ is greater than $k$, then  $s$ intersects  $\ell_i$ to order $> k$ at $c$ (i.e.\ to order $\ge k$ but not exactly to order $k$);
        \item if the multiplicity of $x_0$ at $c$ is nonzero, then $s$ vanishes at $c$. \qedhere
    \end{itemize}
\end{defn}

\begin{prop}\label{vanishing-sections-stratification}
The incidence
\[Z^\cts = \{((w \le x) \in (W_n \le \Hilb(C)^Q), f: C \to V) ~|~ f \in Z_{w \le x}\}\]
is a closed subset of $(W_n \le \Hilb(C)^Q) \times_{W_n} E^\cts$ and therefore defines a stratification of $E^\cts$.
\end{prop}

\begin{proof}
Given \cref{vanishing-sections-vs-sections-of-twisted-bundle}, we want to check that for any continuous choice of divisors $(w \le x) : K \to (W_n \le \Hilb(C)^Q)$ and sections $s : K \times C \to \cE$, the set of $k \in K$ for which $s_k$ takes values in $\ell_i$ on $\supp(x_{k, i} - w_{k, i})$ for each $i$ and vanishes on $\supp(x_0)$ is closed in $K$.
Fixing $i > 0$, we may assume that $x_{k, i} - w_{k, i} \in \Sym^m U$ and $s_k \in \cE|_{U \times U'}$ for some fixed $m$, and some fixed neighborhoods $U \subset C$ and $U' \subset W_n \times \Pic^d$ such that $\cE$ is trivial on $U \times U'$.
Then $s_k$ induces (a continuous choice of) a map $\Sym^m s_k : \Sym^m U \to \Sym^m V$ and therefore the evaluation $(\Sym^m s_k)(x_{k, i} - w_{k, i})$ is a continuous choice $K \to \Sym^m V$.
It suffices to note that $\Sym^m \ell_i \subset \Sym^m V$ is closed. The case of $x_0$ is similar.
\end{proof}

\subsection{Finite dimensional approximation}

As above, let $\cL_d$ be the Poincar\'e bundle on $\Pic^d C$. We defined $E = E^{\cts}$ to be the space parameterizing the data of $L \in \Pic^d C$, $w \in W_n$ and $s \in \Gamma^{\cts}(L \otimes V)$ such that at a multiplicity $k$ point of $w$ labeled by $i$, the section $s$ passes through $\ell_i$ to multiplicity $\ge k$. $E$ comes with a map to $W_n \times \Pic^d$ and for $L \in \Pic^d$ fixed we denote the fiber over $(w, L)$ by $E_{w, L}$.
We want to approximate this bundle of Banach spaces by vector bundles of finite rank, in the sense of \cref{open-intersection-dense-subspace-weak-equivalence} below.
To this end, we will need an elementary lemma, which says that if $\cF \to X$ is a finite rank vector bundle over a paracompact space, with an inclusion (over $X$) into a locally trivial bundle of Banach spaces $\cB \to X$ then the pair $(\cB, \cF)$ is trivializable over any contractible (open) subset.
We provide a proof instead of locating a reference.

\begin{lem} \label{trivialize-fd-inclusions}
Suppose $B$ is a Banach space, $F$ is a finite dimensional vector space, $s : X \to \Emb(F, B)$ is a continuous family of embeddings of $F$ in $B$ and $x_0 \in X$ is a chosen basepoint.
Then there is a neighborhood $U$ of $x_0$ and a continuous choice of automorphisms $t : U \to \Aut(B)$ such that $t(x) \cdot s(x) = s(x_0)$ for each $x \in U$.
If $X$ is additionally paracompact then $U$ can be taken to be any contractible neighborhood of $x_0$.
\end{lem}

\begin{proof}
Since $F$ is finite dimensional, by the Hahn--Banach theorem there is a continuous splitting $p : B \to F$ of $s(x_0)$, ie $p$ is continuous and $p\cdot s(x_0) = \id_F$.
Now $x \mapsto p \cdot s(x)$ is a continuous map $X \to \End(F)$ taking the value $\id_F$ at $x_0$. 
So there is some neighborhood $U_1$ of $x_0$ and a continuous map $r : U_1 \to \GL(F)$ such that $\rank(x) \cdot p \cdot s(x) = \id_F$ for each $x \in U_1$.
Now let $t(x) = \id_B + (s(x_0) - s(x)) \cdot \rank(x) \cdot p \in \End(B)$. Then $t(x_0) = \id_B$ and for each $x \in U_1$,
\[t(x)s(x) = s(x) + (s(x_0) - s(x)) \cdot \rank(x) \cdot p \cdot s(x)  = s(x) + s(x_0) - s(x) = s(x_0) .\]
Since $B$ is Banach, $\Aut(B)$ is open in $\End(B)$ so it suffices to take $U = t^{-1}(\Aut(B))$.

The paracompact case then follows from a standard argument using partitions of unity; see e.g., \cite[proof of Theorem~4.3]{husemoller}.
\end{proof}

The next proposition shows that the homology of finite-dimensional approximations converges to the homology of the space of all sections.   The special case $X = *$ is known as the Palais--Svarc lemma.

\begin{prop} \label{open-intersection-dense-subspace-weak-equivalence}
Let $B$ be a Banach space and $X$ be a locally contractible metric space (in particular paracompact). 
Suppose $\cF_n$ is a chain of subspaces ($\cF_n \subset \cF_{n + 1}$) of $X \times B$ such that:
\begin{itemize}
    \item Each $\cF_n \to X$ is a vector bundle of finite rank.
    \item The union $\bigcup \cF_n$ is fiberwise dense, i.e., for each $x \in X$, the union of fibers $\bigcup \cF_{n, x}$ is dense in $B$.
\end{itemize}
Let $G \subset X \times B$ be open. 
Then $G' := \bigcup G \cap \cF_n \to G$ is a weak homotopy equivalence. 
In fact, for each $s \ge 0$ (and any choice of basepoint in $G \cap \cF_1$):
\[\colim_n \pi_s(G \cap \cF_n) \cong \pi_s(G') \cong \pi_s(G).\]
\end{prop}

\begin{proof}
Using a hypercovering argument (see e.g., \cite{dugger2004topological}, as in the proof of \cref{weak-equivalence-via-hypercovering}), we may assume $X$ to be contractible.
We prove that for every pair of finite complexes $(Y, Y_0)$ and choice of map $f: (Y, Y_0) \to (G, G \cap \cF_n)$, there is some $m \ge n$ and a map $\tilde{f} : Y \to G \cap \cF_m$, homotopic to $f$ rel $Y_0$ and such that $\mathrm{pr}_1 \circ f = \mathrm{pr}_1 \circ \tilde{f} : Y \to X$. 

Since $Y$ is compact, let $\varepsilon > 0$ be such that the $5\varepsilon$ neighborhood of the image of $f$ is contained in $G$.
Let $f = (f_X, f_B)$.
Subdivide $(Y, Y_0)$ such that the image of each simplex of $Y$ has diameter at most $\varepsilon$ and each simplex of $Y$ not contained in $Y_0$ has at least one vertex not in $Y_0$.
For each vertex $v \in Y \setminus Y_0$, since $\bigcup \cF_{m, f_X(v)}$ is dense in $B$, for $m$ sufficiently large we can find
\[\tilde{f}_B(v) \in G_{f_X(v)} \cap B(f_B(v), \varepsilon) \cap \cF_{m, f_X(v)}.\]
Set $\tilde{f}_B$ on $Y_0$ to be equal to $f_B|_{Y_0}$.
Now using \cref{trivialize-fd-inclusions}, trivialize the pair $(X \times B, \cF_m)$.
Then the linear extension of $\tilde{f}_B$ to $Y$ and the linear homotopy with $f_B$ land in $G \cap \cF_m$: each simplex now has diameter at most $3\varepsilon$, so each point moves at most $5\varepsilon$ and therefore remains in $G$; since we've trivialized $\cF_m \hookrightarrow X \times B$, linear combinations of points in $\cF_{m, x_i} \subset B$ are in $\cF_{m, x} \subset B$ for any choice of $x_i, x \in X$.
It suffices to set $\tilde{f} = (f_X, \tilde{f}_B)$.

The above argument shows that for each $s$, the composition
\[\colim \pi_s(G \cap \cF_n) \to \pi_s(G') \to \pi_s(G)\]
is an isomorphism.
Thus it suffices to verify that $\colim \pi_s(G \cap \cF_n) \to \pi_s(G')$ is surjective, which follows from the same argument applied to maps $(Y, Y_0) \to (G', G \cap \cF_n)$.
\end{proof}

\subsection{Semialgebraic approximations}
\label{approximating-subbundles}

We want to apply \cref{BarApproximation} to $Z^\cts$, or rather to a chain of semialgebraic finite-rank approximations. To be more precise, we want to find a sequence \(\cF_k \subset E^{\cts}\) of finite-rank vector bundles over \(W_n \times \Pic^d C\) such that:
\begin{enumerate}
    \item \(E^\alg \subset \cF_1 \subseteq \cF_2 \subseteq \dotsb\).
    \item Each \(Z^\cts \cap \cF_k\) is semialgebraic in the following sense: for fixed \(i\), each \((w_*, L_*) \in W_n \times \Pic^d(C)\) has a (semialgebraic) atlas \(U_\alpha\), and local trivializations \(\cF_k|_U \cong U \times \bbC^{r_k}\), with semialgebraic transition functions, such that each \(Z^\cts \cap \cF_i|U_\alpha\) is semialgebraic.
    \item For a fixed \((w, L) \in W_n \times \Pic^d(C)\), the union \(\bigcup_k \cF_{k, (w, L)}\) of fibers at \((w,L)\) is dense in \(E^\cts_{(w, L)} \cong \Gamma(C, \cE_{d,n}) \).
\end{enumerate}
In the rest of this section we will construct such a sequence, mostly following  \cite[Section 5]{aumonier}, though the ideas can also be traced back to \cite{mostovoy2006stoneweierstrass}.

Fix an (algebraic) embedding of \(C\) in \(\bbC P^N\) for some large \(N\), and let \(M\) be the corresponding restriction of \(\cO(1)\), with ``coordinate'' sections \(z_0, \dots, z_N\).
Let \(\overline{M}\) denote the corresponding \emph{anti-holomorphic} line bundle (i.e.\ one obtained by complex conjugating the transition maps).
We get a semialgebraic trivialization \(\eta \from C \times \bbC \to M \otimes \overline{M}\) given by \(\sum_{i=0}^N z_i \otimes \bar{z}_i\) (and \(\eta^{-1}\) is a semialgebraic choice of Hermitian inner product on \(M\)).

For \(k \ge 0\), let \(\cF_k\) be the image of the composition
\[\phi_k : \Gamma^\hol(\cE \otimes M^{\otimes k}) \otimes \Gamma^{\antihol}(\overline{M}^{\otimes k}) \to \Gamma^\cts(\cE \otimes M^{\otimes k}) \otimes \Gamma^\cts(\overline{M}^{\otimes k}) \to \Gamma^\cts(\cE)\]
where \(\Gamma^{\antihol}\) denotes anti-holomorphic sections, the first map includes (anti-)holomorphic sections among continuous sections and the second map multiplies the sections pointwise and applies \((\eta^{-1})^{\otimes k}\).

\begin{lem}
    The map \(\phi_k\) is injective.
\end{lem}
\begin{proof}
    It suffices to prove that the map
    \[\Gamma^\hol(M) \otimes \Gamma^{\antihol}(\overline{M}) \to \Gamma^\cts(M) \otimes \Gamma^\cts(\overline{M}) \to \Gamma^\cts(M \otimes \overline{M})\]
    is injective. 
    Restricting to a holomorphic disk \(D\) in \(C\) (which we identify with the unit disk in \(\bbC\)), and denoting the set of holmorphic functions \(D \to \bbC\) by \(\cO(D)\), we want to show that \(\cO(D) \otimes \overline{\cO}(D) \to \cC^0(D, \bbC)\) is injective.
    Now it suffices to note that \(f \in \cO(D)\otimes\overline{\cO}(D)\) is determined by its image in \(\bbC[[z]] \otimes \bbC[[\bar{z}]] \hookrightarrow \bbC[[z, \bar{z}]]\), and the coefficients of this image are in turn determined by \(\{\partial^i \overline{\partial}^j f_{\cC^\infty}\}\), where \(f_{\cC^\infty}\) denotes the image of \(f\) in \(\cC^\infty(D, \bbC) \subset \cC^0(D, \bbC)\) and $\partial$, $\bar{\partial}$ are $\left.\frac{\partial}{\partial z}\right|_0$, $\left.\frac{\partial}{\partial \bar{z}}\right|_0$ respectively.
\end{proof}

We have that $\cF_0 = \im \phi_0 = E^\alg$, and $\im \phi_k \subseteq \im \phi_{k+1}$, because $\phi_{k}$ is the composite of $\phi_{k+1}$ with the morphism 
\begin{multline*}
\Gamma^{\hol}(\pi,\cE \otimes M^{\otimes k}) \otimes \Gamma^{\overline{\hol}}(\pi, \ \overline M^{\otimes k}) \\
\to \Gamma^{\hol}(\pi,\cE \otimes M^{\otimes k}) \otimes \Gamma^{\overline{\hol}}(\pi, \ \overline M^{\otimes k}) \otimes \Gamma^{\hol}(C, M) \otimes \Gamma^{\overline{\hol}}(C,\overline M)\\
\to \Gamma^{\hol}(\pi,\cE \otimes M^{\otimes k+1}) \otimes \Gamma^{\overline{\hol}}(\pi, \ \overline M^{\otimes k+1}),
\end{multline*}
defined by $s \otimes t \mapsto \sum_{i = 1}^N z_i s \otimes \overline z_i t$.
The fiberwise density of $\bigcup \cF_n$ in $\Gamma^\cts(\cE)$ is proved the same way as \cite[Lemma~7.2]{aumonier}, which is an application of a suitable generalized Stone--Weierstrass theorem.
Finally, $Z^\cts \cap \cF_k$ is semialgebraic because $Z^\cts \subset E^\cts$ is cut out by linear conditions on the value $s(c) \in \cE_{w, L, c}$ of $s \in E^\cts_{w, L}$ for varying $c \in C$.
When $s \in \cF_k$, in the coordinates $z_i$ and with a choice of an algebraic local trivialization for $\cE$, these values $s(c) \in \cE_{w, L, c} \cong V$ are rational functions in $z_i$ and $\overline{z}_i$.

    \section{Comparing algebraic maps and the semi-topological model}
    \label{sec:Semitopcompare}
           
            In this section, we specialize to our main case of interest:  $V$ is a finite dimensional vector space and $\ell_1, \dots, \ell_r$ are lines corresponding to points $p_1, \dots, p_r \in \bbP(V)$. We let $X = \Bl_{p_1, \dots, p_r} \bbP(V),$ and consider the space of algebraic maps from $C \to X,$ where $C$ is a fixed algebraic curve.  We fix a tuple $(d,n) = (d, n_1, \dots, n_r) \in \bbN^{r+1}$.

            \subsection{Unpointed comparison theorem} 
            
            Let $W_n \subseteq \prod_{i = 1}^r \Sym^{n_i}(C)$ be the open subset consisting of tuples of pairwise disjoint divisors.   We let $\tilde \cM_{d,n}$ be the space parameterizing the data of $L \in \Pic^d(C)$ and $w \in W_n$ and $s \in \Gamma(L \otimes V)$ a continuous section intersecting $\ell_i$ to order exactly $w_i$.  We will prove the following theorem comparing the space of algebraic maps with $\cM_{d,n} := \tilde \cM_{d,n}/\bbC^*$.  As described in \cref{generalstratification}, there is a stratification of the space of algebraic sections of the universal degree $d$ line bundle, $\Gamma_{\alg}(C, \cL \otimes V)$ by $\Hilb(C)^{Q_r}$.  For $y \in \Hilb(C)^{Q_r}$ we define the expected codimension $\gamma(y) := \sum_{i= 1}^r 2(\dim_{\bbC} V - 1) m_{\ell_i}(x) + 2\dim_\bbC(V) m_0(x)$.  As in \cref{combinatorialfunctions} have an associated definition of $\kappa(w < x)$.

            The following is our main criterion for establishing that the map from $\Alg_{d,n}(C,X) \to \cM_{d,n}$ induces isomorphisms on homology groups.  To state it, recall that a map $f: X \to Y$ is \emph{homology $I$-connected} if $H_i(f): H_i(X) \to H_{i}(Y)$ is an isomorphism for $i < I$ and a surjection for $i = I$.  
			 
			\begin{thm}\label{semitopcompare}
			 	  Let $I \in \bbN$. Suppose there is a poset $P \subseteq (W_n < \Hilb(C)^{Q_r})$ that is a closed and downward closed union of finitely many combinatorial types such that 
		          \begin{enumerate}
                        \item $P$ is proper over $W_n$ 
			             
			             \item  For every pair $w < x$ in $P$, and every $y \in (W_n \le Q^{\rJ C})$ such that $x \prec y$ and $y$ is essential the fiber $\Gamma_{{\rm alg}}(C, V \otimes L)_{y}$ is unobstructed, for every line bundle $L \in \Pic^d(C)$. 
                         \item  $P$ contains all types $T$ with $\kappa(T) \leq I$ and all the minimal types of $W_n < \Hilb(C)$.
                \end{enumerate}
                Then the map $\Alg_{d,n}(C, X) \to \cM_{d,n}$ is homology $I$-connected.  
			\end{thm}

            \begin{proof}
            To establish \cref{semitopcompare}, we first make a series of reductions. (For brevity we write $\tilde \Alg_{d,n}$  instead of $\tilde \Alg_{d,n}(C,X)$.) 
                \begin{enumerate} 
                    \item   By the Leray spectral sequence, it suffices to prove that the map of principal $\bbC^*$ bundles from $\tilde \Alg_{d,n}(C, X)$ to $\tilde \cM_{d,n}$ is homology $I$-connected. 
                    
                    \item  Both $\tilde \cM_{d,n}$ and $\tilde \Alg_{d,n}$ map to $W_n \times \Pic^d(C)$.   Again by Leray, it suffices to prove, for basis of open subsets $U \subseteq W_n \times \Pic^d(C)$, that $\tilde \Alg_{d,n}|_U \to \tilde \cM_{d,n}|_U$ is homology $I$-connected.    
                    
                    \item   $\tilde \cM_{d,n}$ is an open subset of bundle of Banach spaces $E^{\rm top} \to W_n \times \Pic^d(C)$, where the fiber above $(w,L)$ is the space of sections of $L$ intersecting $\ell_i$ holomorphically to order $w_i$.   This bundle contains the space of all algebraic sections $E^{\rm alg}$  that intersect $\ell_i$ to order $\geq w_i$.    By the construction in \cref{approximating-subbundles} there is a basis of (semialgebraic) open subsets $U \subseteq W_n \times \Pic^d(C)$ such that on every $U$ there is a sequence of semialgebraic finite-dimensional sub-bundles $E^{\beta} = \cF_\beta \subseteq E^{\rm top}|_U$ containing $E^\alg|_U$ with $\cup_{\beta} E^{\beta}$ fiberwise dense.  By \cref{open-intersection-dense-subspace-weak-equivalence}, it suffices to prove that the map $$\tilde \Alg_{d,n}|_U \to \tilde \cM_{d,n}^{\beta}|_U := E^{\beta} \cap \tilde \cM_{d,n}|_U  $$ is homology $I$-connected all $\beta$ sufficiently large.
                    
                    \item  We have that $\tilde \Alg_{d,n}|_U  = E^{\rm alg} \cap \tilde \cM_{d,n}^\beta|_U .$  Hence by \cref{gysinPD}, the pushforward map   $$H_i(\tilde \Alg_{d,n}|_U)\to  \tilde H_i(\tilde \cM_{d,n}^\beta|_U)$$ is Poincar\'e dual to the Gysin map $$H^{\dim E^{\rm alg} - i}_c(\tilde \Alg_{d,n}|_U) \to  H^{\dim E^{\beta} - i}_c( \tilde \cM_{d,n}^\beta|_U)$$ induced by the Thom class for the inclusion of vector bundles $E^{\rm alg}|_U \subseteq E^{\beta}$.  So it suffices to prove that this Gysin map is an isomorphism for $i < I$ and surjection for $i = I$.

                    \item Let $K^{\rm alg} = E^{\rm alg}- \tilde \Alg_{d,n} |_U$  be the complement of the space of algebraic maps.  Similarly let $K^{\beta} := E^\beta - \cM_{d,n}^\beta$.  By the long exact sequence in compactly supported cohomology, it suffices to prove that the Gysin map $C^*_c(K^{\rm alg})[\dim E^{\rm alg}] \to  C^*_c(K^{\beta})[\dim E_{\beta}]$  is $(I+1)$-connected.  
                \end{enumerate}
                
            Now we are in a position to apply the results of the previous sections. 
            The stratification $Z^{\cts}$ of $E^{\cts}$ defined in \cref{topological_stratification_definition} restricts to a stratification of $E^\beta$ by $(W_n \le \Hilb(C)^{Q_r})$, denoted by $Z^\beta$.  As described in the running example~\ref{mainex1}, we have a corresponding stratification of the space of algebraic sections $\Gamma(C, \cL_d \otimes V)$ by $\Hilb(C)^{Q_r}$ inducing a stratification of $E^{\rm alg}|_U$ by $\Hilb(C)^{Q_r}$.  
            
            Restricting to $P_U$, we have a commutative diagram of bar constructions

            $$ \begin{tikzcd}
K^{\beta}                           & K^{\rm alg} \arrow[l]                           \\
{{\rm B}(U,P, Z^{\beta})} \arrow[u] & {{\rm B}(U,P, Z^{\rm alg})} \arrow[u] \arrow[l],
\end{tikzcd} $$
    where the vertical maps are surjections, because $P$ contains the minimal types of $W_n < \Hilb(C)^Q$.  This square induces a commutative diagram on compactly supported cochains
    $$
    \begin{tikzcd}
        C^*_c(K^{\beta})[\dim E^\beta]   \arrow[d]                     & C^*_c(K^{\rm alg}) [\dim E^{\rm alg}] \arrow[l]  \arrow[d]                         \\
        C^*_c({{\rm B}(U,P, Z^{\beta})}) [\dim E^\beta]  & C^*_c({{\rm B}(U,P, Z^{\rm alg})}) [\dim E^{\rm alg}] \arrow[l],
    \end{tikzcd}
    $$
  where the horizontal arrows are Gysin maps.   
  
  We obtain that the lower horizontal map is a quasi-isomorphism by applying \cref{GysinIso} with $E = E^\beta, Z = Z^\beta, E' = E^{\rm alg}$ and $Z' = Z^{\rm alg}$. To do so, we need to verify for $(w < x) \in (W_n < \Hilb(C)^Q)$ essential,
  \begin{equation}\label{essentialintersection}
      Z_{w<x}^\beta \cap E_w^\alg = Z_{w<x}^\alg.
  \end{equation} This follows from the \cref{topological_stratification_definition} of $Z^\beta$, and the description of essential pairs $w < x$ in our situation, as in \cref{essential_example}. 

  We prove that both of the vertical maps are $(I+2)$-connected by applying \cref{BarApproximation} twice (with $R$ equal to the set of $w < x$ such that $w < \sat(x)$ is essential and $P$ as given).    For the rightmost vertical map, the hypotheses of \cref{BarApproximation} follow from our assumptions on $P$. For the left vertical map, the unobstructedness hypothesis (2) follows from \eqref{essentialintersection}, the fact that codimensions can only decrease on intersection, and that $\gamma(w<x)$ is the maximal possible codimension of $Z_{w<x}^{\cts} \subseteq E_{w}^{\cts}$.  
  
  Therefore the upper horizontal map is  ($I+1$)-connected,  and so we are done.  \end{proof}

  \begin{remark}\label{remark:Whysemitop?}
In an earlier stage of this project, we believed that our methods would allow us to compare the fiber of $\tilde \Alg_{d,n}(C,X)$ over $(D_i)_{i = 1}^r \in W_n$ to a space of sections of a jet bundle;  namely the space consisting of pairs $L \in \Pic^d C$ and continuous non-vanishing sections $s \in \Gamma(\rJ( L \otimes V))$ satisfying the condition:
	\begin{itemize}
	\item	if $c \in C$ is a multiplicity $e$ point of $D_i$, then the jet $s(c)$ intersects $\ell_i$ to order exactly $e$ (in the sense that $s(c)$ is a Taylor expansion of such a function).  
	\end{itemize}	
However, this is not the case.  To illustrate the difference between the space of sections of the jet bundle and the model $\cM_{d,n}$ we use, let us specialize the case where $C = \bbC, V = L = \bbC$ and $D$ is the divisor of degree one supported at the origin.  (Here $r = 1$ and $\ell_1 = 0 \subseteq \bbC$).

Let $\cI \subseteq J^{1} \bbC$ be the incidence condition stating that a function $s$ vanishes only at the origin to order exactly $1$.  We have that $J^1 \bbC = \bbC^3$ and $$\cI = ((\bbC-0) \times (\bbC-0) \times \bbC )\cup ( 0 \times 0 \times \bbC - 0),$$  which is not open or closed (or even locally compact).  Furthermore $\Gamma(\bbC, \cI)$ consists of a pair of continuous functions $a: \bbC \to \bbC$ and $b: \bbC \to \bbC$ such that $a$ vanishes exactly at $0$ and $b$ is non-vanishing at $0$.  The data of $a$ and $b$ are independent, and $a$ is determined up to homotopy by its restriction to the unit circle, while $b$ is determined by its value at $0$:  so we have that $\Gamma(\bbC,\cI)  \simeq \Top(S^1, \bbC-0) \times (\bbC-0)$.

On the other hand, consider the space of functions $s: \bbC \to \bbC$ that intersect $0$ holomorphically to order exactly one at the origin.  By a sort of Alexander trick (see \eqref{Tayloragreement}),  this space retracts to the subspace of functions that agree with a nonzero $\bbC$-linear function on the unit disk. We may further retract to the subspace of nonzero $\bbC$-linear functions, identified with $\bbC-0$. 

The key difference between these two cases is that in the first there is no relationship between the ``formal Taylor expansion" of $s$ at $0$ and its actual values.  In contrast, the second space is equivalent to the space of pairs consisting of an injective linear function $z \mapsto bz$ and a function $f: S^1 \to \bbC-0$ and a homotopy between them:  via the map that takes $s$ to its Taylor coefficient at $0$ and its restriction to $S^1$ and the homotopy provided by its restriction to the unit disk.  

The situation does not improve when $C$ is compact. For instance, if $C = \bbP^1$ and $L = \cO(d)$ with $d$ large, the space of holomorphic sections with exactly one zero of order $1$ at $0 \in \bbP^1$ is empty whereas the relevant subspace of sections of the jet bundle is always non-empty.
\end{remark}

    \subsection{Pointed variant} \label{subsec:pointedvariant}
            We now consider the pointed case where $C$ and $X$ are equipped with distinguished base points $*_C \in C$ and $*_X \in X- \cup_{i = 1}^r E_i$.  We  take the same approach, but with a modified poset.  Accordingly let $$W_{n,*} :=  \left( \prod_{i = 1}^{r} \Sym^{n_i} C\right) \times *_C \subseteq  \left( \prod_{i = 1}^{r} \Sym^{n_i} C\right) \times C \subseteq \Hilb(C)^{Q_{r+1}}.$$ 

            Let $\overline R \subseteq \Hilb(C)^{Q_r}$ be the closed subposet consisting of divisors $D_{\ell_1}, \dotsc, D_{\ell_{r+1}}, D_0$ such that $D_{\ell_{r+1}} - D_{0}$ is supported at $*_C$ with multiplicity $\leq 1$.  We again have a stratification $\Gamma_{\alg}(C,\cL \otimes V)$ by $\Hilb(C)^{Q_{r+1}}$ restricting to a stratification by $\overline R$  The expected codimension of $y \in \Hilb(C)^{Q_{r+1}}$ is $\gamma(y):=  \sum_{i 1}^{r+1} 2(\dim_{\bbC} V - 1) m_{\ell_i}(y) + 2\dim_\bbC(V) m_0(y)$.
            
           We define $R_*: =  (W_{n,*} < \Hilb(C)^{Q_{r+1}}) ~ \cap ~ (W_{n,*} \le \overline R)$ and $R$ to be the subset of $R_*$ consisting of $(w \le x)$ such that $(w \le \sat(x))$ is \emph{essential}.  Then we have the following pointed variant of \cref{semitopcompare}.

             \begin{thm}\label{pointedsemitopcompare}
			 	  Let $I \in \bbN$. Suppose there is a poset $P \subseteq R$ that is a closed and downward closed union of finitely many combinatorial types such that 
		          \begin{enumerate}
			             \item  For every pair $w < x$ in $P$, and every $y \in  Q_{r+1}^{\rJ C}$ such that $x \prec y$ and $y$ is essential the fiber $\Gamma_{{\rm alg}}(C, V \otimes L)_{y}$ is unobstructed, for every line bundle $L \in \Pic^d(C)$. 
                         \item  $P$ contains all types $T$ of $R$ with $\kappa(T) \leq I$ and all the minimal types of $R$
                \end{enumerate}
                Then the map $\Alg_{d,n, *}(C, X) \to \cM_{d,n,*}$ is homology $I$-connected.  
			\end{thm}    
   
                The proof of \cref{pointedsemitopcompare} is identical to that of \cref{semitopcompare}.  We note that $R_*$ is initial in $W_{n,*} < \Hilb(C)^{Q_{r+1}}$:  if $x = (D_{\ell_1}, \dots, D_{\ell_{r+1}}, D_0)$ then $\eta_R(x) = (D_{\ell_1}, \dots, D_{\ell_{r+1}}', D_0)$  where the divisor $D_{\ell_{r+1}}'$ equals $D_0$ away from $*_C$ and has multiplicity $\min(1, \mult_{*_C} (D_{\ell_{r+1}}- D_0))$ at $*_C$. Therefore $R$ is also initial.

\section{The semi-topological model and spaces of positive maps}
\label{sec:posmapscompare}

Let $X = \Bl_{p_1, \dots, p_r} \bbP(V)$. We defined $\cM_{d,n}$ to be the space parameterizing the data of $L \in \Pic^d(C)$, $w \in W_n$ and a choice up to scalar multiple of a continuous section $s \in \Gamma(L \otimes V)$ intersecting $\ell_i$ to order exactly $w_i$. The results of the previous section can be used establish that the canonical map $$\Alg_{d,n}(C,X) \to \cM_{d,n}$$ induces an isomorphism on  homology in a given range.     (For simplicity, we will only consider the unpointed case in this section, but all of the results hold with the same arguments in the pointed case).

The purpose of this section is to prove that $\cM_{d,n}$ is weakly equivalent to a space parameterizing continuous maps $C \to X$ of positive intersection multiplicity.  (For notational convenience, we fix $d\in \bbN, n \in \bbN^r$ for the remainder of the section and simply write $\cM$).  We let $\cT$ be the subspace of $W_n \times \Top(C, X)$ consisting of $(w \in W_n, f: C \to X)$ satisfying:
\begin{itemize}
	\item the projection of $f$ to $\bbP(V)$ is degree $d$
	\item  $f\inv(E_i)$ is discrete, and $f$ has positive local intersection multiplicity at every point of $f\inv(E_i)$
		\item at a multiplicity $k$ point of $w_i$, the map $f$ has intersection multiplicity $k$ with $E_i$.
\end{itemize}
Here \emph{local intersection multiplicity} is defined as follows:
\begin{defn}\label{intersection-multiplicity}
    Suppose $f : M^d \to N^n$ is a continuous map of oriented manifolds, $Z^{n-d} \subset N$ is an oriented submanifold and $p \in M$ is an isolated point of $f^{-1}(Z)$. 
    Then for a standard neighborhood $V$ of $f(p)$ (i.e.\ such that $(V, V \cap Z) \cong (\bbR^n, \bbR^d)$) with its implied orientation, $f$ induces a map $H_d(M, M - p) \to H_d(V, V - Z)$ and we call the corresponding integer the local intersection multiplicity of $f$ (with $Z$) at $p$.
\end{defn}

The main result of this section is the following. 

\begin{thm}\label{thm:posmapsoverbase}
	There is a canonical weak homotopy equivalence $\cM \to \cT$.
\end{thm}

   We construct the map $\cM \to \cT$ in two stages.  First there is a projection map,  $\cM \to \cM'$ where $\cM'$ consists of pairs $w \in W_n$ and continuous degree $d$ maps $f: C \to \bbP^n$  such that $f$ intersects $p_i$ to order exactly $w_i$.

Next we construct a homeomorphism $\cM' \to \cT'$, where $\cT' \subseteq \cT$ consists of continuous degree $d$ maps $f: C \to X$ such that $f$ intersects $E_i$ holomorphically to order $w_i$. The map $\cM' \to \cT'$ is induced by the strict transform operation defined in the following subsection.

\subsection{Strict transform}  Let $U \subseteq \bbC$.  Consider the blowdown map $\pi: \Bl_0 \bbC^v \to \bbC^v$.  

\begin{prop}\label{strictransform}
	Let $w \in \Sym^m U$.  Then postcomposition with $\pi$ defines a homeomorphism between the following spaces of  functions:
	\begin{itemize}
		\item  continuous functions $U \to \Bl_0 \bbC^v$ intersecting the exceptional divisor holomorphically to order exactly $w$.  
		\item  continuous functions $U \to \bbC^v$ intersecting the origin holomorphically to order exactly $w$.  
	\end{itemize}
\end{prop}
\begin{proof} Given $(f_1, \dots, f_{v}) : U \to \bbC^{v}$ vanishing to order exactly $w$ we may write $f_i$ uniquely as $f_i = \rho_w \tilde f_i$, where $\rho_w$ is the unique monic degree $m$ polynomial vanishing at $w$.  Then $$(f_1, \dots, f_{v}) \times [\tilde f_1: \dots : \tilde f_{v}]$$ is a map to the blowup.   Recall that the blowup $\Bl_0 \bbC^{v} \subseteq \bbC^v \times \bbP(\bbC^v)$ has standard coordinate charts for $i = 1, \dots n$ given by $$\Bl_0(\bbC^v)-\bbV(\tilde x_i) \to \bbC^v \qquad (x_1, \dots, x_v) \times [\tilde x_1: \dots : \tilde x_v] \mapsto (\tilde x_1/\tilde x_i, \dots, x_i, \dots, \tilde x_{v}/\tilde x_i).$$ 
Thus locally at any point of $U$ where $\tilde f_i$ does not vanish, in the $i$th chart $U$ is given by $$(\tilde f_1/ \tilde f_i, \dots  ,f_i,  \dots,  \tilde f_{{v}}/\tilde f_i),$$ hence intersects the exceptional divisor holomorphically to order exactly $w$.  
\end{proof}

The proposition immediately globalizes to the case where $(\bbC^n, 0)$ is replaced by a complex manifold and finite set of points $(M, \{p_1, \dots, p_r\})$, yielding a homeomorphism between spaces of:

\begin{itemize}
	\item  continuous functions $C \to \Bl_{p_1, \dots, p_r} M$ intersecting the exceptional divisor $E_i$ holomorphically to order exactly $w_i$.  
		\item  continuous functions $U \to M$ intersecting the point $p_i$ holomorphically to order exactly $w_i$.  
\end{itemize}
The inverse homeomorphism is called the \emph{strict transform},  yielding a homeomorphism $\cM' \to \cT'$.

\subsection{Comparing $\cM$ and $\cT$}  So far, we have constructed a map $\cM \to \cT$. We now prove that this map is a weak homotopy equivalence.  To do so, we will cover $\cT$ and $\cM$ by open subsets of the following form.

We will use a basis of open neighborhoods of $W_n \subseteq \prod_{i = 1}^r \Sym^{n_i}(C)$ which we call distinguished opens.
\begin{defn}
    Fix a metric on $C$.  A point $w \in W_n$ corresponds to a finite configuration of points in $T \subseteq C$, together with a label for each $t\in T$ of the form $(i_t , m_t)$,  where $i_t= 1, \dots, r$ and $m_t \in \bbN_{> 0}$.  For any $\epsilon > 0$ with $2 \epsilon < \underset{t\neq t' \in T} \min d(t,t')$, we say that the \emph{distinguished open subset} of radius $\epsilon$ around $w$ is the open subset  $$\rB(\epsilon,w) := \prod_{t \in T} \Sym^{m_t}(\rB(\epsilon,t)) \subseteq W_n,$$ where $\rB(\epsilon, t)$ is the open ball of radius $\epsilon$ around $s$.
\end{defn}

\begin{defn}
	Given a finite union of closed balls $K \subseteq C$  labelled by $\{1, \dots, r\}$,  we let $\cT(K)$ denote the subspace consisting of functions $f: C \to Y$ that map the balls of $K$ labelled by $i$ to the tubular neighborhood ${\rm Tub}(E_i)$.   For a subset $R \subseteq W_n$, we let $\cT_{R} (K):= \cT_R \cap \cT(K)$, where $\cT_R$ is the preimage of $R$ under the projection $\cT \to W_n$.
	
	We define $\cM_{R}(K)$ to be the pre-image of $\cT_{R}(K)$.   For a distinguished open subset $V \subseteq W_n$,  centered on $w$ of radius $\epsilon$, we write $V \subseteq K$ if the radius $\epsilon$ neighborhood of each point of $w$ labelled by $i$ is contained in a ball of $K$ labelled by $i$. 	
\end{defn}

First we establish a local to global principle.
\begin{prop}\label{weak-equivalence-via-hypercovering}
	To prove that $\cM \to \cT$ is a weak equivalence, it suffices to prove that  $\cM_{V}(K) \to \cT_V(K)$ is a weak equivalence for every $V \subseteq K$ such that $V$ is a distinguished open and $K$ is a finite union of closed balls labelled by $\{1, \dots, r\}$.
\end{prop}

\begin{proof}
	The subspace $\cT_{V}(K) \subseteq \cT$ is open by definition of the compact open topology.   Furthermore, for any $(w,f) \in \cM$, there is an $\epsilon$ such that the image of the closed ball of radius $\epsilon$ around every point of $f\inv(E_i)$ contained in ${\rm Tub}(E_i)$.  Taking $K$ to be the union of these balls and $V$ to be a distinguished neighborhood of radius $\epsilon/2$ of $w$, we have that $(w,f) \in \cT_V(K)$.  So $\cT_{V}(K)$ form an open cover.
	
	We have the identities $\cT_{V_1}(K_1) \cap  \cT_{V_2}(K_2) = \cT_{V_1 \cap V_2} (K_1 \cup K_2)$ and $\cup_{\alpha} \cT_{V_\alpha} (K) = \cT_{\cup_{\alpha} V_{ \alpha}}(K)$.  Since the distinguished open subsets form a basis of $W_n$, it follows from these identities that any finite intersection of subsets of the form $\cT_{V}(K)$ is covered by opens of the same form.  Therefore, we may form a hypercover $\cU_\cT$ of $\cT$ by disjoint unions of the open subsets $\cT_{V}(K)$.  Taking preimages, we obtain a hypercover $\cU_\cM$ of $\cT=N$ by disjoint unions of $\cM_{V}(K)$.  
	
	By the main result of Dugger--Isaksen \cite{dugger2004topological}, we have that $\underset{\Delta \op}{\rm hocolim}~\cU_\cM \to ~\cM$ and $\underset{\Delta \op} {\rm hocolim}~ \cU_{\cT} \to \cT$ are weak equivalences.  Therefore, if $\cM_K(V) \to \cT_K(V)$ are weak equivalences for every $K,V$ then $\cM \to \cT$ is a weak equivalence.
\end{proof}

Next we retract onto fibers.  
\begin{prop}\label{Retract_To_Fibers}
	Let $K \subseteq C$ be a finite union of closed balls labelled by $1, \dots, r$. Let $V$ be a distinguished open neighborhood of $w \in W_n$   satsifying $V \subseteq K$.    Then the inclusions $\cT_{w}(K) \to \cT_{V}(K)$ and $\cM_{w}(K) \to \cM_{V}(K)$ are homotopy equivalences.  
\end{prop}
\begin{proof}

First we construct a deformation retraction of $\cT_V(K)$ onto $\cT_w(K)$.  Recall that $w$ corresponds to a labelled subset $T \subseteq C$ and $V$ is determined by a radius $\epsilon$.  Let  $t \in T$ be labelled by $i$ with multiplicity $m$.  We will construct a deformation retraction from the space of pairs $w' \in \Sym^{m}(B(\epsilon,t))$,  $f: \overline {\rB(\epsilon, t)} \to  {\rm Tub}(E_i)$ with multiplicities of $f\inv(E_i)$ prescribed by $w'$, to the subspace of pairs that satisfy $w' = m t$.  This deformation retraction will fix the values of $f$ on the boundary of $\overline {B(\epsilon, t)}$,  so that by gluing together the retractions for each element of $T$ we obtain the desired retraction of $\cT_V(K)$.

To construct the retraction, we use a variant of the Alexander trick.  Identify ${\rm Tub}(E_i)$ with the normal bundle $\pi : N_{E_i} \to E_i$, and $B(\epsilon,t)$ with the unit disk.    We proceed in three steps.  First, given $w',f$ make $f$ radially constant on the points of absolute value $\geq 1/2$,  via $$f_u(x) = \begin{cases} f(\frac{x}{1-u}) & \text{if } |x| \leq 1-u\\ f(x/|x|) & \text{if } |x| \geq {1-u} \end{cases}$$ for $u \in [0,1/2]$.  (We define $w'_u$ to be the scaling of $w'$ by $1-u$). 

 Second we make $\pi \circ f(x) = \pi \circ f(0)$ for all $x$ with $|x| \leq 1/2$,  using the radial path from $\pi(f(x))$ to $\pi(f(0))$ and lifting to $N_{E_i}$.  More precisely we lift the homotopy  $$\overline f_u(x) = \begin{cases} \pi (f((1-u)x) & \text{if $|x| \leq 1/2$}  \\ \pi (f((1-u)x + xu(2|x| - 1)) & \text{if $|x| \geq 1/2$}\end{cases}$$ for $u \in [0,1]$ to $N_{E_i}$. (To perform the lift we choose a metric on $N_{E_i}$ and write $N_{E_i}$ as the cone on the unit circle bundle.  Then the universal path lifting for the unit circle bundle provides a continuous way of lifting paths of maps $X \to E_i$ to paths of maps $X \to N_{E_i}$ such that  the distance from the zero is preserved section).

Finally  we apply the Alexander trick to the radius $1/2$ disk $$f_u(x) = \begin{cases}  (1-u)f(\frac{x}{1-u}) & \text{if } |x| \leq 1/2 (1-u)\\  2|x|f(\frac{x}{2|x|}) & \text{if } 1/2(1-u) \leq |x| \leq 1/2 \\  f(x) & \text{if } |x| \geq 1/2\end{cases}$$ for $u \in [0,1]$ and $f_1(0) = 0$.  This three step process yields the desired retraction,  because in the second step $w'$ is fixed and in the third step $w'$ is scaled down to the origin.

Next we construct a deformation retraction of $\cM_V(K)$ onto $\cM_w(K)$.   Similarly to above, we will construct a retraction from the space of pairs $(w', s)$, where $w' \in \Sym^{m}(B(\epsilon,t))$,  and $s$ is a section of $L\otimes V$ over $\overline{B(\epsilon, t)}$ that is $w'$-holomorphic and such that the multiplicity of $f\inv(E_i)$ is prescribed by $w'$, to the subspace of pairs that satisfy $w' = m t$.  Gluing together these local retractions then yields a global retraction.

  To describe the retraction we choose coordinates $V \iso \bbC^v$ such that $\ell_i = \bbC \times 0^{v-1}$, and $s$ corresponds to sections $s_1, \dots, s_{v}$ of $L$.   Because the map to $\bbP^{v-1}$ associated to $s$ takes $\overline {B(\epsilon,t)}$ to the neighborhood $\pi({\rm Tub}(E_i))$ of $p_i$, we have that $s_1$ nonvanishing.    Therefore $s_1$ defines a continuous trivialization of $L|_{\overline{ B(t,\epsilon)}}$  (which varies continuously in $w$ and $s_1$) that we use to identify $s_1, \dots, s_v$ with functions $U \to \bbC$.   (Here $s_1$ is the constant function $1$.)   Then because $s_i$ vanish to order $w$ for $i \geq 2$, we have that $s_i = \tilde s_i p_w$ for unique nonvanishing functions $\tilde s_i$ . Recall that $p_w(z)$ is the polynomial in $z$ with roots at $w$, namely  $\prod_{c \in w} (z - c)$. 
  
  We retract $w$ to the origin radially, letting $w_u = (1-u) w$ for $u \in [0,1]$,  and defining $s_{i,u} = \tilde s_i ~ p_{w,u}$  for $i > 1$  where \begin{equation}p_{w,u}(z):=
     \begin{cases} \prod_{c \in w}( z - (1-u)c) & \text{if } |z| \leq \max(1-u, 1/2) \\ 
  			  \prod_{c \in w}( z - |z|c) & \text{if } u \leq 1/2, ~ |z| \geq 1-u  \\
   			 \prod_{c \in w}( z - (1+2u(|z|-1))c) & \text{if } u \geq 1/2, ~ |z| \geq 1/2.
      \end{cases} \label{tworing} \end{equation}
In words, we break up the unit disk into two regions:  the interior disk of radius $\max(1-u, 1/2)$ and the exterior ring.   On the interior, we scale down the roots of $p_w$ to the origin,  and on the exterior we interpolate the roots linearly so that $p_{w,u}$ is unchanged on the boundary of the ball.  
\end{proof}

\begin{prop}
	Let $w \in W_n$, and $K$ be a union of labelled closed balls satisfying $w \in K$.  Then $\cM_w(K) \to \cT_w(K)$ is a homotopy equivalence. 
\end{prop}
\begin{proof}
	We show that each of the maps $\cM_w(K) \to \cM_w'(K)$ and $\cT_w'(K) \to \cT_w(K)$ are weak homotopy equivalences.   (This suffices because by \cref{strictransform} the map $\cM_w'(K) \to \cT_w'(K)$ is a homeomorphism with inverse given by post-composition with the projection $Y \to \bbP(V)$).
	
	First consider $\cM_w(K) \to \cM_w'(K)$.  The target parameterizes $w$-holomorphic maps $C \to \bbP(V)$ intersecting $p_i$ to the order prescribed by $w$.     Now $w$-holomorphic maps $C \to \bbP(V)$ correspond to pairs consisting of a $w$-holomorphic line bundle $L$ and $s \in \Gamma(L \otimes V)$ a non-vanishing $w$-holomorphic section,  considered up to isomorphism.  (Here $(L,s)$ is isomorphic to $(L', s')$ if there is a continuous, and therefore $w$-holomorphic, isomorphism of line bundles $L \to L'$ taking $s$ to $s'$).  The fiber $F$ over an element $\cM_w'$ corresponding to a $w$-holomorphic line bundle $L$ and section $s$ parameterizes the data of:
	\begin{itemize}
		\item  $\tilde L$ a holomorphic line bundle 
		\item  $\tilde s \in \Gamma(V \otimes \tilde L)$ a section intersecting $\ell_i$ holomorphically order precisely $w_i$
		\item $f: \tilde L \to L$ a continuous isomorphism of $\bbC$ line bundles, taking $\tilde s$ to $s$, 
	\end{itemize}
considered up holomorphic isomorphism of pairs $(\tilde L, \tilde s)$.   This data is uniquely determined by the isomorphism $f$,  so $F$  parameterizes pairs consisting of $\tilde L$ a holomorphic bundle and $f: \tilde L \to L$ a continuous isomorphism (considered up to holomorphic isomorphisms of $\tilde L$).     In other words it is the space of holomorphic structures on $L$.  The map $\cM_w(K) \to \cM_w'(K)$ is a fiber bundle because there is a cover of $\Top(C, \bbP(V))$ by open subsets $U$ such that the  universal bundle $\ev^* \cO_{\bbP(V)}(1) |_{U \times C}$ is isomorphic to $\pi^* \cO_{C}(d)$.    Here $\ev$ is the evaluation map $ \Top(C, \bbP(V)) \times C \to \bbP(V)$ and $\pi$ is the projection to $C$ and $\cO_C(d)$ is the unique degree $d$ continuous line bundle on $C$.   A choice of isomorphism between the universal bundle over $U \times C$  and $\cO_C(d)$ yields a trivialization over $U \cap \cM_w'(K)$,  identifying the fiber of every point with the space of holomorphic structures on $\cO_C(d)$. 

   To show that $\cM_w(K) \to \cM_w'(K)$ is an equivalence, we prove that the fiber $F$ is contractible. First $F$ fibers over the space of $C^{\infty}$ structures on $L$, which is isomorphic to $\Top(C,\bbC^*)/C^{\infty}(C, \bbC^*)$ because there is one $C^{\infty}$ degree $d$ complex line bundle up to isomorphism. 
   Since the inclusion of $C^\infty$ functions to continuous functions is an equivalence so the space of $C^{\infty}$ structures on $L$ is contractible.  Equipping $L$ with a Riemannian metric, we have that holomorphic structures on $L$ correspond to unitary connections on $L$.   Therefore the fiber over a given $C^\infty$ structure on $L$ is identified with the space of unitary connections (a torsor for the space of $(1,0)$ forms with values in $L$). (See the related discussion at \cite[565]{atiyahbott}). This establishes that $\cM_w(K) \to \cM_w'(K)$ is an equivalence.

	Second we show $\cT_w'(K) \to \cT_w(K)$ is a weak equivalence.  Choosing disjoint disks around each point of $w$  inside of the balls of $K$ (i.e. the closure of a distinguished neighborhood $V \subseteq K$), we reduce to the following local statement.  Let $i \in \{1, \dots, r\}$ $m \in \bbN$.   For a map $D \subseteq \bbC$ the unit disk, and a degree $m$ map $g: \del D \to {\rm Tub}(E_i)-E_i$ the inclusion of mapping spaces between 
    \begin{multline*}
    \{f : D \to {\rm Tub}(E_i), ~ f|_{\del D} = g,~ f\inv(E_i) = 0  ~|\\
    \text{$f$ intersects $E_i$ holomorphically to order $m$ at $0$}\}
    \end{multline*}and $$\{f : D \to {\rm Tub}(E_i), ~ f|_{\del D} = g,~ f\inv(E_i) = 0 \} $$ is an equivalence.

	We fiber both spaces over $E_i$ via the map $f \mapsto f(0)$ to reduce to spaces of maps satisfying the additional condition that $f(0) = p$.   Accordingly we analyze the homotopy type of $$G:= \{f : D \to {\rm Tub}(E_i), ~ f|_{\del D} = g,~ f\inv(E_i) = 0, f(0) = p\}.$$ 
	As in the proof of \cref{Retract_To_Fibers} we may deformation retract  onto the subspace consisting of maps satisfying $\pi(f(x)) = p$ for $|x| \leq 1/2$.   By the Alexander trick applied to the normal slice $N_p$ to $p$, we may retract further to the space of functions that are radial for $|x| \leq 1/2$ (in other words functions satsifying $f(rx) = r f(x)$ for $|x| \leq 1/2$ and $r \in [0,1]$).   This shows that $G$ is  equivalent to $$\{g': S^1 \to N_p - 0, ~  h \text{ homotopy between $g'$ and $g$ }  \}.$$ 
			Next we consider the subspace of $H \subseteq G$ consisting of functions that intersect $E_i$ holomorphically to order $m$ at $0$.  Identifying $N_p$ with an open ball in $\bbC$, we  will show that this subspace is equivalent to $$ \{g': S^1 \to N_p - 0, ~  h \text{ homotopy between $g'$ and $g$ } ~|~ g'(z) = a z^d \text{ for some $a\in \bbC^*$}\}$$ and the inclusion is the natural one.   Because the space of degree $m$ maps $\Top_m(S^1, \bbC - 0)$ retracts to the subspace of maps of the form $z \mapsto a z^m$ where $a \in \bbC^*$, establishing this will complete the proof.   
			
			Since the holomorphicity condition is only imposed in the normal direction to $E_i$, as above we may first deformation retract to the subspace of maps such that $\pi(f(x)) = p$ for $|x| \leq 3/4$.   For such a map, the restriction of $f$ to the radius $3/4$ ball $D_{3/4}$ is determined by the induced map to the normal slice $\tilde f: D_{3/4} \to N_p$.  Then we retract further to the subspace of functions  $f$ that satisfy $\pi(f(x)) = p$ for $|x| \leq 1/2$ and  such that the induced map $\tilde f: D_{1/2} \to N_p \subseteq \bbC$ agrees with its degree $m$  Taylor polynomial.   To do this, we use a rescaling of the following deformation retraction from the space $$\{ j: D \to \bbC ~|~ f\inv(0) = 0,  \text{ $j$ intersects zero holomorphically to order exactly $m$ }\}$$ to the subspace of functions that agree with their degree $m$ Taylor polynomial at the origin on the ball of radius $1/2$: set $j_u(z) = \lambda^{-m} j(\lambda z)$ with $\lambda = \lambda(u, z)$ given by
		
	\begin{equation}\lambda(u, z):= \begin{cases} 1 - u & \text{if } |z| \leq  1/2 \,, \\ 
  			  u (2|z|-1) + (1-u) & \text{if } |z| \geq 1/2 \,.
         \end{cases} \label{Tayloragreement} \end{equation}
  In words, we break the disk into an interior disk of radius $1/2$ and an exterior annulus.  On the interior disk use the $j_u(z) :=(1-u)^{-m} j((1-u) z)$, and on the exterior annulus we linearly interpolate between $j$ and the values on the interior.    The key point is that, because $j$ vanishes to order exactly $m$ at the origin, $\lim_{t \to 0} t^{-m} j(t z)$ equals the degree $m$ Taylor polynomial of $j$.  Because of this, $j_0(z)$ is well-defined and agrees with its degree $m$ Taylor expansion at the origin on the radius $1/2$ disk.
\end{proof}

\subsection{Relating $\cT$ and $\Top^+$}  In this subsection, we recall the dependence of the constructions on a choice of $(d,n)$.  Recall that $\cT_{d,n}$ was defined to be a subset of $W_n \times \Top_{d,n}(C,X)$ given by forgetting the choice of divisors $w \in W_n$.  

\begin{prop}\label{prop:forgetdivisors}
    The projection map $\cT_{d,n} \to \Top_{d,n}(C,X)$ is a homoemorphism onto its image.  
\end{prop}
\begin{proof}
The projection is bijective, because given a function $f: C \to Y$ that intersects $E_i$ positively for $i = 1, \dots, r$ we can recover $D_i$ as the unique divisor supported at $f\inv(E_i)$, whose multiplicity at a point of $f\inv(E_i)$ is the local intersection multiplicity of $f$ with $E_i$.  

	To show that the inverse is continuous,  fix $w \in W_n$ corresponding to a finite labelled subset $T \subseteq C$ and a distinguished open subset of radius $\epsilon$ containing $w$.  Let $f \in {\rm Top}_+(C, Y)_{d,n}$ be a function whose intersection multiplicities with $E_i$ are those specified by $w$.  Consider the union of circles of radius $\epsilon/2$ around each $t \in T$, and the open subspace  consisting of functions $f' \in  {\rm Top}_+(C, Y)_{d,n}$ such that $f'$ differs from $f$ by at most $\delta$ on each circle.  For $\delta$ sufficiently small, we have that for every $f'$ in this subspace, its restriction to each radius $\epsilon/2$ circle is homotopic to the restriction of $f'$ (considered as a map to $Y-\cup_{i} E_i$).  In particular, the intersection multiplicity of $f'$ on the interior of each circle agrees with that of $f'$.  Since $f'$ only intersects $E_i$ positively, it follows that the divisor associated to $f'$ is contained in the distinguished open subset of $w$ of radius $\epsilon/2$.  Hence the inverse is continuous.
\end{proof}

Since by definition, the image of $T_{d,n}$ under the projection map is $\Top_{d,n}^+(C,X)$, we have that \cref{prop:forgetdivisors,thm:posmapsoverbase} yield the following.  (We also include the pointed variant, which follows from arguments identical to the ones in this section.)

\begin{thm} \label{thm:posmaps}
    There is a weak equivalence $\cM_{d,n} \simeq \Top_{d,n}^+(C,X)$ and $\cM_{d,n,*} \simeq \Top_{d,n,*}^+(C,X)$.
\end{thm}

\section{Unobstructedness}
\label{sec:unobstructedness}
In this section, we verify the hypotheses of \cref{semitopcompare} in several cases, and consequently deduce our main results.  Establishing these hypotheses boils down to proving that certain section spaces are unobstructed.  There is a range where section spaces are always unobstructed by a simple application of Riemann--Roch, which we consider in \cref{RRrange}.  If the points $p_1, \dots, p_{\dim V}$ are in linearly general position, then this range can be improved slightly: we do this in \cref{GPrange}.  Finally in \cref{deg5} we consider the case of the degree 5 del Pezzo surface, where we apply the results of \cref{GPrange} and the symmetry under $S_5$ to improve the range to a dense open subset of the ample cone.   

Throughout this section we let $V$ be a vector space of dimension $v \geq 3$, containing lines $\ell_1, \dots, \ell_r$  corresponding to $p_1, \dots, p_r \in \bbP(V)$.  We let ${Q_r}$ denote the poset $\{0, \ell_1, \dots, \ell_r, V\}$, and let $C$ be a fixed smooth proper genus $g$ curve.

\subsection{Riemann--Roch}\label{RRrange} First we apply Riemann--Roch directly to construct a poset to which we can apply Theorem \ref{semitopcompare}.   We make our statements using the multiplicity functions $m_{\ell_j}, m_0$ of \cref{combinatorialfunctions}.  We will set $m_\ell:= \sum_{j = 1}^r m_{\ell_j}$.  Recall that given $x \in {Q_r}^{\rJ C}$ the expected codimension of $\Gamma(C, L \otimes V)_x$ in $\Gamma(C, L \otimes V)$ is $\gamma(x) = 2(\dim V m_0(x) + (\dim V - 1) m_\ell(x))$.   We have the following simple consequence of Riemann--Roch.

\begin{prop}\label{RRoch}
	Let $C$ be a smooth projective genus $g$ curve.  Let  $x \in {Q_r}^{\rJ C}$.  If $d - m_0(x) - m_{\ell}(x) \geq 2g - 1$, then $\Gamma_{Q_r}(C, L \otimes V)_x$ is unobstructed.  
\end{prop}

\begin{prop}\label{posetconstruction} Let $d \in \bbN ,n \in \bbN^r$, and let $$I = d  -  \sum_{i =1}^r n_i -2g$$
 Let $W \subseteq \prod_{i = 1}^{r} \Sym^{n_i} C$ be the locus of pairwise disjoint divisors.  Consider the subset $P \subseteq (W < \Hilb(C)^{Q_r})$ consisting of $w<x$  that satisfy $m_{0}(x) + m_{\ell}(x) \leq I + \sum_{i= 1}^r n_i$, then 

		\begin{enumerate}
			\item $P$ is proper over $W$ and downward closed
   \item for all $w<x \prec y$, with $(w < x) \in P$ and every $L \in \Pic^d(C)$ the fiber $\Gamma_{{Q_r}}(C,L \otimes V)_{y}$ is unobstructed. 
			\item $P$ contains every type $T$ with $\kappa(T) \leq I$.
		\end{enumerate}
		
\end{prop}
\begin{proof} 
	For $(1)$, it suffices to show that the subset of $\Hilb(C)^{Q_r}$ consisting of $x$ satisfying $m_0(x) + m_{\ell}(x) \leq I$ is closed for every $j = 1, \dots v$.  This follows from \cref{closednessprop}.
	
	Condition (2) follows from \cref{RRoch}.
		For  (3) suppose that $w < x \in \cS_T$ with $\kappa(T) \leq I$.  
		Then $\rank(w<x) \leq I$ by \cref{ranklemma}.  Now $\rank(w<x) = 2 m_0(x) + m_\ell(x) - \sum_{i = 1}^{r} n_i$.  Since $m_0(x) \geq 0$, we obtain that $m_0(x) + m_\ell(x) \leq I - \sum_{i=1}^r n_i $. 
\end{proof}

\begin{lem}\label{closednessprop}
 Let $J \in (1/2, 1]$, and $I \in \bbR$. Then the set of $x \in \Hilb(C)^{Q_r}$ such that $m_{ 0}(x) + J m_{\ell}(x)  \leq I$ is  compact and downward closed.  
\end{lem}
\begin{proof}
Downward closure is immediate, because the function $E:= m_{ 0} + J m_{\ell}$ is increasing.  To show compactness it suffices to show that the set of $x$ such that $E(x) \leq I$ is closed: since if $m_{0}(x) + J m_{\ell}(x)  \leq I$   then $x \in (\Hilb(C)_{\leq \max(I, I/J)})^{r+1}$.  

For closedness, suppose that $g_1, g_2: {Q_r} \to \bbN$ are two functions.   We claim that $$m_{0}(g_1) + J m_{\ell}(g_1) + Jm_{0}(g_2) + m_{\ell}(g_2)  \geq m_{0}(g_1 + g_2)  + J m_{\ell}(g_1 + g_2),$$ (i.e. that $E(g_1 + g_2) \leq E(g_1) + E(g_2)$).  This claim establishes closedness, because it implies that collection of combinatorial types satsifying the the condition $m_{ 0}(T) + J m_{\ell}(T) \leq I$ is downward closed with respect to the partial order $\leq_+$.  

Because $g \leq \sat(g)$ for all $g$, we have that $\sat(g_1 + g_2) \leq \sat(\sat(g_1) + \sat(g_2))$.  Since $E(g) = E(\sat(g))$, we have  $E(g_1 + g_2) \leq E(\sat(g_1) + \sat(g_2))$,  so it suffices to establish the claim with $g_1, g_2$ replaced by their saturations.  

So suppose that $g_1, g_2$ are saturated and correspond to the chains  $(a_1 - a_2) \ell_i + a_2$ and $(b_1 - b_2) \ell_j + b_2$  for $a_1 \geq a_2$ and $b_1 \geq b_2$ and $i,j \in \{1, \dots r\}$.   If $i = j$,  then $E(g_1 + g_2) = E(g_1) + E(g_2)$.    Otherwise, suppose without loss of generality that $a_2 + b_1 \geq a_1 + b_2$.  Then the saturation of $g_1 + g_2$  corresponds to the chain $$(a_1 + b_1 - a_2 - b_1)  \ell_{k} + (a_1 + b_2) 0,$$ where $k = i$ or $j$  (depending on $\max(a_1, b_1)$).   Thus we have $$E(g_1 + g_2) = (a_1 + b_1 - a_2 - b_1)J + (a_1 + b_2) $$ and $$E(g_1) + E(g_2) = (a_2 - a_1 + b_2 - b_1) J  + a_2 + b_2.$$  Simplifying, the inequality $E(g_1 + g_2) \leq E(g_1) + E(g_2)$ reduces to $(2J - 1)(a_1 - a_2) \geq 0$,  which holds by our assumption on $J$.
\end{proof}

\begin{lem} \label{ranklemma}
	We have that $\rank(x) \leq \kappa(x)$ for all $x \in \Hilb(C)^{Q_r}$.
\end{lem}
\begin{proof}
	If $x_1 \prec x_2 \in {Q_r}^{\rJ C}$, then it follows from direct computation that $\kappa(x_2) \geq \kappa(x_1) + 1$.  Indeed, observe that $\kappa(x)  = (2v-3)m_{\ell}(x) + (2v-2) m_{0}(x) + 2 |{\rm Supp}(x)|$ and $|{\rm Supp}(x)|$ can increase by at whereas at least one of $m_0(x)$  or $m_\ell(x)$ must increase by one.  Since $v \geq 3$ the inequality follows.
 
 Thus a maximal chain of elements of ${Q_r}^{\rJ C}$ ending in $\sat(x)$ (which must have length $\rank(x)$), we obtain that $\kappa(x) \geq \rank(x)$.
\end{proof}

\subsection{Points in general position}\label{GPrange}

Write $v = \dim V$.   In this subsection, we assume that the points $p_1, \dots, p_{\min(v,r)} \in \bbP(V)$ are in linearly general position. In this case, we have the following improvement on Riemmann--Roch.   

\begin{prop}\label{unobstructedness}
 Let $d \in \bbN$ and let  $x \in {Q_r}^{\rJ C}$, with  $m_{\ell_i}(x) = n_i$.   If $- m_{\ell_j}(x) + m_\ell(x)  \leq d-m_0(x) + 2 -2g$ for all $j = 1, \dots, \min(v,r)$, then $\Gamma_{Q_r}(C,L^v)_x$ is unobstructed for any $L \in \Pic^d(C)$.
	
	\end{prop}
\begin{proof} For notational simplicity we assume that $r \geq v$, the case where $r < v$ is similar.

	Let $L$ be a degree $d$ line bundle on $C$, and let $D_{\ell_1}, \dots D_{\ell_r}, D_0$ be divisors specifying an element $x \in {Q_r}^{\rJ C}$.  Because $x$ is saturated, $U_i := D_{\ell_i} - D_0 \in \Sym^{m_{\ell_i}(x)}(C)$ are a collection of pairwise disjoint divisors.  
 
    We must show that $\Gamma_{Q_r}(C, L^{v})_x$ has the expected dimension.  Without loss of generality, we take $p_1, \dots, p_{v}$ to be the points corresponding to the lines generated by the standard coordinate vectors $e_i \in \bC^{v},  i = 1, \dots, v$.  Then $\Gamma_{Q_r}(C, L^{v})_x$  consists of tuples of sections $s = (s_i)_{i = 1}^{v} \in \Gamma(L^{v}(-D_0))$  such that $$s_i(U_j) = 0,\quad  i,j \in \{1, \dots, v\}, ~i \neq j$$  and $s (U_k)  \in \ell_k$ for $k = v+1, \dots, r$.  Thus $\Gamma_{Q_r}(C, L^{v})_x = \Gamma(\cF)$ where $\cF$ is the fiber product of the following diagram of sheaves
\begin{equation*}\begin{tikzcd}                                                   &  \overset{r}{\underset{b = v+1}\prod} \ell_b  \arrow[d] \\
\overset{v}{\underset{j = 1}\prod}~ L\left(-D_0-\overset{v+1}{\underset{i = 1, i \neq j}\sum} U_i\right) \arrow[r, "p"] & \overset{r}{\underset{b = v+1}\prod} L^{v}(-D_0)|_{U_b}.               
\end{tikzcd}
\end{equation*}
Now there is a short exact sequence of sheaves $$0 \to \prod_{j = 1}^{v} L(-D_0-\overset{r}{\underset{i = 1, i \neq j}\sum} U_i)\to  \cF \to \prod_{b = v+1}^r {\ell_b}\to 0,$$ because the left hand term is the kernel of $p$,  and $p$ is surjective as a map of sheaves. If $- m_{\ell_j} (x) + m_{\ell(x)}  \leq d - m_0(x) + 2 -2g$ for all $j = 1, \dots, r$, then by Riemann--Roch  we have that $H^1(L(-D_0-\overset{r}{\underset{i = 1, i \neq j}\sum} U_i)) = 0$ for all $j = 1, \dots, v$, so by the long exact sequence in cohomology $H^1(\cF)  = 0$  hence $H^0(\cF)$ has the expected dimension. 

 In general, we have that  $\dim H^1(\cF) \leq \sum_{j = 1}^{\min(v,r)} \left(d - 2g + 2 - \sum_{i = 1, i \neq j}^r n_i \right)$.
\end{proof}

Because of \cref{unobstructedness}, we may construct the following poset. 

\begin{prop}\label{genposposetconstruction} Let $d \in \bbN ,n \in \bbN^r$, and let $$I = d - \sum_{i =1}^r n_i -2g + \min \{n_1, \dots, n_v\} $$
where we take $n_j = 0$ if $r < j \le v$.
 Let $W \subseteq \prod_{i = 1}^{r} \Sym^{n_i} C$ be the locus of pairwise disjoint divisors.  Consider the subset $P \subseteq (W < \Hilb(C)^{Q_r})$ consisting of $w<x$  that satisfy $$ m_{0}(w<x) + \sum_{i\neq j} m_{\ell_i}(w<x) \leq I$$ for all $j = 1, \dots, v$.    Then

		\begin{enumerate}
			\item $P$  proper over $W$ and downward closed
			\item for all $w<x \prec y$, with $w < x \in P$ and every $L \in \Pic^d(C)$ the fiber $\Gamma_{{Q_r}}(C,L \otimes V)_{y}$ is unobstructed. 
			\item $P$ contains every type $T$ with $\kappa(T) \leq I$.
		\end{enumerate}
		
\end{prop}
\begin{proof} 
	For $(1)$, it suffices to show that the subset of $\Hilb(C)^{Q_r}$ consisting of $x$ satisfying $m_0(x) + m_{\ell}(x) -m_{\ell_j}(x) \leq I$ is closed for every $j = 1, \dots v$.  Considering the projection forgetting points labelled by $\ell_j$, this follows from \cref{closednessprop}.
	
	Condition (2), follows from \cref{unobstructedness},  together with the fact that if $y \succ x$ then $$d -m_0(y) + m_{\ell}(y) - m_{\ell_j}(y) \geq d -m_0(x) + m_{\ell}(x) - m_{\ell_j}(x) -1 \geq d - I - 1 \geq 2g-2,$$
 since $d- I$ is nonnegative.

		For (3) suppose that $w < x \in \cS_T$ where $\kappa(T) \leq I$.  
		Then $\rank(w<x) \leq I$ by \cref{ranklemma}.  Now $\rank(w<x) = 2 m_0(w<x) + m_\ell(w<x)$.   Since $m_{\ell_i}(x) - n_i \geq 0$ for all $i$,  it follows that $m_0(x) + \sum_{i \neq j,i = 1}^{r} (m_{\ell_i}(x) - n_i) \leq I$ for every $j = 1, \dots, v$.
\end{proof}

\subsection{Proof of main theorem and corollary}
With \cref{posetconstruction} and \cref{posetconstruction} in place, we are now in a position to establish \cref{thm:blowupThm} and \cref{non-positivecorollary}.

\begin{proof}[Proof of \cref{thm:blowupThm}]
    We first consider the case where the points $p_1, \dots p_{\min(v,r)}$ are assumed to be in general position.    Given $(d,n) \in \bbZ^{r+1}$, by \cref{genposposetconstruction}  there is a subposet $P \subseteq (W < \Hilb(C)^{Q_r})$ that satisfies the hypotheses of \cref{semitopcompare} for $I(d,n) = \min_{j = 1, \dots,\min(v,r)}(d + n_j -  \sum_{i =1}^r n_i) -2g$.  Therefore the map $\Alg_{d,n}(C,X) \to \cM_{d,n}$ is homology $I(d,n)$-connected.  So by \cref{thm:posmaps}, the map $\Alg_{d,n}(C,X) \to \Top^+_{d,n}(C,X)$ is homology $I(d,n)$-connected.  

    In the pointed case, we apply the \cref{genposposetconstruction} to the tuple $(d,n_1, \dots, n_r, 1)$.  Let $Q_{r+1}$ be the poset $\{0, \ell_1, \dots, \ell_r, \ell_{r+1}\}$ where $\ell_{r+1} \subseteq V$ is the line corresponding to the basepoint of $X$.  Then \cref{genposposetconstruction} yields a subposet of $P \subseteq (W_{n+1} < \Hilb(C)^{Q_{r+1}})$, which is unobstructed and contains all combinatorial types $T$ such that $\kappa(T) \leq I(d,n) - 1$.  Now consider the intersection of $P$ with $R \subseteq (W_{n,*} < \Hilb(C)^{Q_{r+1}})$  (see \cref{subsec:pointedvariant} for Definition).  Given a saturated combinatorial type $T$ of $P$ we have that $\cS_T \cap R$  is a (possibly empty) union of strata of pointed combinatorial types $T'$.  Furthermore,  $\kappa(T') \geq \kappa(T)$ because the expected codimension $\gamma$ is unchanged while $\dim \cS_{T'} \leq \dim \cS_{T}$ and $r(T') \leq r(T)$. Therefore the hypotheses of \cref{pointedsemitopcompare} apply to $P \cap R$ with $I = I(d,n) - 1$.  So as above,  $\Alg_{d,n,*}(C,X) \to \Top^+_{d,n,*}(C,X)$ is $I(d,n) - 1$ connected.

    To rephrase these connectivity statements in the form stated in the introduction, suppose that $\alpha = (d',n') \in \bbN^{r+1}$ is a class satisfying $d' + n'_j -  \sum_{i =1}^r n'_i > 0$ for all $j = 1, \dots, \min(v,r)$.   Let $M_\alpha = \min(d' + n'_j -  \sum_{i =1}^r n'_i)$.  Then for $k \in \bbN$ we have that $I(kd',kn') = M_{\alpha} k - 2g$ hence for $i \leq M_{\alpha}k - 2g - 2$ the induced maps on $H_i$ are isomorphism.

    When the points $p_1, \dots, p_{\min(v,r)}$ are not assumed to be in general position, we apply the same argument, using \cref{posetconstruction} instead of \cref{genposposetconstruction}.
\end{proof}

\begin{proof}[Proof of \cref{non-positivecorollary}]
	Let $C$ be a curve.  Up to homotopy, the space $\Top_{(e,l_1, \dots, l_r)}^+(C,X)$ is independent of $e$. More precisely fix a disc around $*_C$ to obtain a  map $\pi: C \to S^2 \vee C$ (given by collapsing the boundary of the disk), and choose a map $\eta:S^2 \to X-(\cup_{i=1}^r E_i)$ that projects to a degree $1$ map to $\bbP^m$ and satisfies $\eta(0) = \eta(\infty) = *_X$.   Wedging with $\eta$ defines a stabilization map $$\Top_{(e,l_1, \dots, l_r)}^+(C,X) \to \Top_{(e+1,l_1, \dots, l_r)}^+(C,X) \qquad f \mapsto (\eta \vee f)\circ \pi.$$  
We may think of $\eta$ as a homotopy from constant map $S^1 \to *_X$ to itself.  Let $\overline \eta$ be the reversed homotopy.  Then wedging with $\overline \eta$ defines a destabilization map that is homotopy inverse to $\eta$,  because the the map $S^2 \to S^2 \vee S^2 \to^{\eta \vee \overline \eta} X-\cup_{i} E_i$ is null-homotopic.

Furthermore, up to homeomorphism $\Top_{(e,l_1, \dots, l_r)}^+(C,X)$  is independent of the choice of points $p_1, \dots, p_r$.  

When  $C = \bbP^1$ and all of the points $p_1, \dots, p_r$ lie on a hyperplane, and $$\min_{j}(e  -l_1 - \dots - l_r, l_j) > i$$ the work of Boyer--Hurtubise--Milgram \cite{boyer1999stability}  shows that $$\Alg_{(e, l_1, \dots, l_r)}(\bbP^1,X) \to \Top_{(e,l_1, \dots, l_r)}(\bbP^1, X)$$ induces an isomorphism on homology in degrees $\leq i$.  (Take $N$ to be the group of translations of that affine space which is the complement of the hyperplane. Then $c_0(X) = 1$ because $X_\infty$ has normal crossings).   Considering the diagram induced on homology by 

\begin{center}
\begin{tikzcd}
{\Alg_{(e+j , l_1, \dots, l_r),*}(\bbP^1,X)} \arrow[r] & {\Top^+_{(e+j , l_1, \dots, l_r),*}(\bbP^1,X)} \arrow[r]                               & {\Top_{(e+j , l_1, \dots, l_r),*}(\bbP^1,X)}                               \\
  
 & {\Top^+_{(e , l_1, \dots, l_r),*}(\bbP^1,X)} \arrow[u, "-\vee \eta^{\vee j}"] \arrow[r] & {\Top_{(e , l_1, \dots, l_r),*}(\bbP^1,X),} \arrow[u, "-\vee \eta^{\vee j}"]
\end{tikzcd}
\end{center}
we obtain that $\Top_{(e , l_1, \dots, l_r),*}^+(\bbP^1,X)\to \Top_{(e , l_1, \dots, l_r),*}(\bbP^1,X)$ induces an isomorphism on $H_i$ for $i \leq \min(l_j)$.  By \cref{thm:blowupThm} we obtain the result. 	\end{proof}

\subsection{The degree 5 del Pezzo Surface}\label{deg5}

Let $X$ be the unique degree $5$ del--Pezzo surface over $\bbC$.   In this section we show that for almost every ample class $\alpha \in \rN_1(X)$, that as $k \to \infty$ there is a bar construction model for the homology of $\Alg_{k \alpha}(C,X).$ First we show that there is a way of writing $X$ as a blowup, such that the ample class $\alpha$ satisfies certain equalities.  

\begin{prop}
	Let $\alpha \in H_2(X, \bbZ)$ be an ample class.   Then there are four distinct points $p_1, p_2, p_3, p_4 \in \bbP^2$, no three lying on an line and an isomorphism $X \iso {\rm Bl}_{p_1, \dots, p_4} \bbP^2$  such that $\alpha = (d,n_1,n_2,n_3, n_4)$  with $n_{i} + n_{j} + n_4 \leq d$  for all choices of distinct $i,j  \in \{1, 2,3\}$ and $n_i \geq 0$ for all $i = 1, \dots, 4$.  If $\alpha$ is such that there are no disjoint $-1$-curves $E, E'$ with $E' \cdot \alpha = E \cdot \alpha$, we may choose this isomorphism such that $n_{i} + n_{j} + n_4 < d$
	\end{prop}
\begin{proof}
		Fix four points in $p_1, \dots p_4 \in  \bbP^2$, no three lying on a line.  It is classical that $X$ is isomorphic to $\Bl_{p_1, \dots, p_4} \bbP^2$.  In terms of this isomorphism, write $\alpha = (e,u_1,u_2, u_3, u_4)$.  By relabelling the points, we may assume that $u_1 \geq u_2 \geq u_3 \geq u_4$.   Now consider the Cremona transformation on $p_1, p_2,p_3$  defining an automorphism of $\Bl_{p_1, \dots, p_4} \bbP^2$.   On $H_2(X)$ it acts by 
        \begin{multline*}
            (e,u_1,u_2,u_3,u_4) \mapsto (d,n_1,n_2,n_3,n_4) \\
            := (2e - u_1 -u_2 - u_3, e - u_2 - u_3, e - u_1 - u_3,  e- u_2 - u_3, u_4).
        \end{multline*}  Now direct calculation shows that $u_i \geq u_4$  implies that $d  \geq n_j + n_k  + n_4$  for $\{i,j,k\} = \{1,2,3\}$,  with equality if and only if $\alpha$ lies in the hyperplane $u_3 = u_4$.  Because $\alpha$ is ample $n_i \geq 0$ for all $i$. 
\end{proof}

Applying \cref{thm:blowupThm}, we immediately obtain the following theorem, as well as \cref{thm:delPezzo}.
\begin{thm}
    Let $C$ be a compact Riemann surface.  For any ample class $\alpha$ such that there are no two disjoint $-1$ curves on $X$ with $E' \cdot \alpha = E  \cdot \alpha$ and any basepoint $*_X$ not contained in the locus of exceptional curve.  Then the maps $$H_i(\Alg_{k\alpha}(C,X)) \to H_i(\Top^+_{k\alpha}(C,X)) \qquad H_i(\Alg_{k\alpha,*}(C,X)) \to H_i(\Top^+_{k\alpha,*}(C,X))$$ induce isomorphisms for $i \leq k N_\alpha - 2g -2$ where $N_\alpha$ is a constant depending on $\alpha$
\end{thm}

\printbibliography

\end{document}